\documentclass[12pt]{amsart} 

\usepackage[OT2, T1]{fontenc}

\usepackage{amscd}
\usepackage{amsmath}
\usepackage{amssymb}
\usepackage{mathrsfs} 			
\usepackage{units}
\usepackage[all]{xy}

\usepackage{algorithm}
\usepackage{algorithmic}

\def\frk{\frak}               

\def\mm{{\frk m}}

\def\Phi{{\frk n}}
\def\Phi{{\frk N}}
\def\opn#1#2{\def#1{\operatorname{#2}}} 

\opn\projdim{proj\,dim} \opn\injdim{inj\,dim} \opn\rank{rank}
\opn\depth{depth} \opn\sdepth{sdepth} \opn\fdepth{fdepth}
\opn\grade{grade} \opn\height{height} \opn\embdim{emb\,dim}
\opn\codim{codim}  \opn\min{min} \opn\max{max}

\opn\Tr{Tr} \opn\bigrank{big\,rank}
\opn\superheight{superheight}\opn\lcm{lcm}
\opn\trdeg{tr\,deg}
\opn\reg{reg} \opn\lreg{lreg} \opn\ini{in} \opn\lpd{lpd}
\opn\size{size}
\opn\div{div} \opn\Div{Div} \opn\cl{cl} \opn\Cl{Cl}
\opn\Spec{Spec} \opn\Supp{Supp} \opn\supp{supp} \opn\Sing{Sing}
\opn\Ass{Ass} \opn\Min{Min}
\opn\Ann{Ann} \opn\Rad{Rad} \opn\Soc{Soc}
\opn\Im{Im} \opn\Ker{Ker} \opn\Coker{Coker} \opn\Am{Am}
\opn\Hom{Hom} \opn\Tor{Tor} \opn\Ext{Ext} \opn\End{End}
\opn\Aut{Aut} \opn\id{id}  \opn\deg{deg}

\opn\nat{nat}
\opn\pff{pf}
\opn\Pf{Pf} \opn\GL{GL} \opn\SL{SL} \opn\mod{mod} \opn\ord{ord}
\opn\Gin{Gin} \opn\Hilb{Hilb}
\opn\aff{aff} \opn\con{conv} \opn\relint{relint} \opn\st{st}
\opn\lk{lk} \opn\cn{cn} \opn\core{core} \opn\vol{vol}
\opn\link{link} \opn\star{star}
\opn\gr{gr}

\def\pot#1#2{#1[\kern-0.28ex[#2]\kern-0.28ex]}

\opn\dirlim{\underrightarrow{\lim}}
\opn\inivlim{\underleftarrow{\lim}}
\let\to=\rightarrow

\def\Implies{\ifmmode\Longrightarrow \else
        \unskip${}\Longrightarrow{}$\ignorespaces\fi}
\def\implies{\ifmmode\Rightarrow \else
        \unskip${}\Rightarrow{}$\ignorespaces\fi}
\def\iff{\ifmmode\Longleftrightarrow \else
        \unskip${}\Longleftrightarrow{}$\ignorespaces\fi}

\let\:=\colon
\newtheorem{Theorem}{Theorem}[]
\newtheorem{Lemma}[Theorem]{Lemma}
\newtheorem{Corollary}[Theorem]{Corollary}
\newtheorem{Proposition}[Theorem]{Proposition}

\theoremstyle{definition}

\newtheorem{Remark}[Theorem]{Remark}

\newtheoremstyle{subsection-tweak}
   {11pt}
   {3pt}%
   {}
   {}%
   {\bfseries}
   {}%
   {.5em}
   {\thmnumber{\@{#1}{}\@{#2}.}%
    \thmnote{~{\bfseries#3.}}}    

\newcounter{numberingbase}

\theoremstyle{subsection-tweak}
\newtheorem{bpp}[Theorem]{}
\newtheorem{bppt}[numberingbase]{}
\newcommand{\bbpp}{\begin{bpp}}
\newcommand{\eepp}{\end{bpp}}
\newcommand{\bbppt}{\begin{bppt}}
\newcommand{\eeppt}{\end{bppt}}

\theoremstyle{theorem}

\theoremstyle{definition}

\newcommand{\val}{\mathrm{val}}		

\providecommand{\qxq}[1]{\quad\text{#1}\quad}

\newcommand{\tst}{\textstyle}

\let\epsilon\varepsilon
\let\phi=\varphi
\def\qed{\ifhmode\textqed\fi
      \ifmmode\ifinner\quad\qedsymbol\else\dispqed\fi\fi}
\def\textqed{\unskip\nobreak\penalty50
       \hskip2em\hbox{}\nobreak\hfil\qedsymbol
       \parfillskip=0pt \finalhyphendemerits=0}
\def\dispqed{\rlap{\qquad\qedsymbol}}

\opn\dis{dis}
\def\pnt{{\raise0.5mm\hbox{\large\bf.}}}

\opn\Lex{Lex}

\begin{document}

\title{Pure transcendental, immediate valuation ring extensions as limits of smooth algebras}

\author{ Dorin Popescu}

\dedicatory{In the honour  of Michael Artin.}

\address{Simion Stoilow Institute of Mathematics of the Romanian Academy,
Research unit 5, P.O. Box 1-764, Bucharest 014700, Romania,}

\address{University of Bucharest, Faculty of Mathematics and Computer Science
Str. Academiei 14, Bucharest 1, RO-010014, Romania,}

\address{ Email: {\sf dorin.m.popescu@gmail.com}}

\begin{abstract} We show that a pure transcendental, immediate extension of valuation rings  $V\subset V'$  containing a field  is a filtered union of  smooth $V$-subalgebras of $V'$.
\end{abstract}

\maketitle

\section*{Introduction}

An immediate extension of valuation rings is an extension inducing trivial extensions on residue fields and group value extensions.  The theorem stated below \cite[Theorem 2]{P} gives on immediate extensions of valuation rings a form of 
Zariski's Uniformization Theorem \cite{Z}.

\begin{Theorem}\label{T0} If $V\subset V'$ is an immediate extension of valuation rings   containing $\bf Q$
then $V'$ is a filtered direct limit of smooth $V$-algebras. 
\end{Theorem}

In the Noetherian case a morphism of rings is a filtered direct limit of smooth algebras iff it is a regular morphism (see \cite{Po0}, \cite{S}). This was important for a positive solution of some Artin's conjectures \cite{Ar}.
 An extension of the above theorem stated in  \cite[Theorem 2]{P1} is the following theorem. 

\begin{Theorem} \label{T1} Let $V\subset V'$ be an extension of valuation rings containing $\bf Q$, $\Gamma\subset \Gamma'$ the value group extension of $V\subset V'$  and $\val: K'^{*}\to \Gamma'$ the valuation of $V'$. Then $V'$  is a filtered direct limit of smooth $V$-algebras if and only if  the following statements hold
\begin{enumerate}

\item for each $q\in \Spec V$ the ideal $qV'$ is prime,

\item  for any prime ideals $q_1,q_2\in \Spec V$ such that $q_1\subset q_2$ and height$(q_2/q_1)=1$  and any $x'\in q_2V'\setminus q_1'$ there exists $x\in V$ such that $\val(x')=\val(x)$,  where $q_1'\in \Spec V'$ is the prime ideal corresponding to the maximal ideal of $V_{q_1}\otimes_V V'$, that is the maximal prime ideal of $V'$ lying on $q_1$.
\end{enumerate}
\end{Theorem} 
 
  It is well known that if the fraction field extension of an immediate extension $V\subset V'$ is finite and the characteristic of $V$ is $>0$, then $V'$ fails to  be  a filtered direct limit of  smooth $V$-algebras as shows  \cite[Example 3.13]{Po1} inspired from \cite{O} (see also \cite[Remark 6.10]{Po1}). However,
   B. Antieau, R. Datta \cite[Theorem 4.1.1]{AD} showed that every perfect valuation ring of characteristic $p>0$ is a filtered union of its smooth ${\bf F}_p$-subalgebras. This result is an application of \cite[Theorem 1.2.5]{T}, which relies on some results from \cite{J}. 

In \cite[Theorem 1]{P2} we state the following theorem.

\begin{Theorem}\label{T2} Let $V\subset V'$ be an  immediate extension of valuation rings containing a field  and  $K\subset K'$ the fraction field extension. If $K'/K$ is algebraic then $V'$ is a filtered
 union of its complete intersection $V$-subalgebras  of finite type.
\end{Theorem}

A {\em complete intersection} $V$-algebra is a local $V$-algebra of type $C/(P)$, where $C$ is a localization of a polynomial $V$-algebra of finite type and $P$ is a regular sequence of elements of $C$.

One goal
of this paper is to show that Theorem \ref{T0} holds  in a special case of positive characteristic.

\begin{Theorem}\label{T3} Let  $V\subset V'$ be an immediate  extension of valuation rings   containing  a field and $K\subset K'$ its fraction field extension. If $K'=K(x)$ for some algebraically independent system of elements $x$ over $K$ 
then $V'$ is a filtered union of smooth $V$-subalgebras of $V'$. 
\end{Theorem}

The proof of Theorem \ref{T3} is given  using Theorem \ref{key}, which  is based on Propositions \ref{p}, \ref{p'}. A preliminary step of the hard proof of Theorem \ref{key} is given in Lemma \ref{1} using Corollary \ref{co} (and so Proposition \ref{p}).  A consequence of Theorem \ref{T3} is the next corollary, which is a form of   Zariski's Uniformization Theorem in a special case of
positive characteristic.
\begin{Corollary}\label{C} Let $V$ be a  valuation ring containing its residue field $k$   with a finitely generated value group $\Gamma$ and $K$ its fraction field. Assume that  $K=k(x,y)$ for some  algebraically independent elements $x=(x_i)_{i\in I},y=(y_j)_{j\in J}$ over $k$ such that $\val(y)$ is a basis in $\Gamma$. 
 Then $V$ is a filtered  union of its smooth $k$-subalgebras.
\end{Corollary}

Another goal is
 to state  the following theorem extending Theorem \ref{T2}.

\begin{Theorem}\label{T4} Let $V\subset V'$ be an  immediate extension of valuation rings containing a field.   Then $V'$ is a filtered
 union of its complete intersection $V$-subalgebras of finite type.
\end{Theorem}

\begin{Corollary}\label{C1} Let $V$ be a  valuation ring containing its residue field $k$   with a free value group $\Gamma$, for example if $\Gamma$ is finitely generated.  
 Then $V$ is a filtered  union of its complete intersection $k$-subalgebras of finite type.
\end{Corollary}

Theorem \ref{T4} follows from Theorems \ref{T2}, \ref{T3} and the proof of  Corollary \ref{C1} goes as the proof of Corollary \ref{C} using Theorem \ref{T4} instead of Theorem \ref{T3}. In positive characteristic, Corollary \ref{C1} is  a form of   Zariski's Uniformization Theorem in  the frame of complete intersection  algebras.

We owe thanks to some anonymous referees who hinted  some mistakes in some  previous forms of Lemmas \ref{o2}, \ref{ka0}, the proofs of Proposition \ref{p}, Theorem \ref{key} and  had some useful comments on Lemmas \ref{ka1}, \ref{ka2} and Proposition \ref{p'}.

\vskip 0.3 cm

 \section{Some extensions of an Ostrowski's lemma}

Let $V$ be a valuation ring, $\Gamma$ its value group, $\val$ its valuation, $\lambda$  a fixed limit ordinal  and $v=\{v_i \}_{1\leq i < \lambda}$ a sequence of elements in $V$ indexed by the ordinals $i$ less than  $\lambda$. Then $v$ is \emph{pseudo convergent} if 

$\val(v_{i} - v_{i''} ) < \val(v_{i'} - v_{i''} )     \ \ \mbox{for} \ \ 1\leq i < i' < i'' < \lambda$
(see \cite{Kap}, \cite{Sch}).
A  \emph{pseudo limit} of $v$  is an element $w \in V$ with 

$ \val(w - v_{i}) < \val(w - v_{i'}) \ \ \mbox{(that is,} \ \ \val(w -  v_{i}) = \val(v_{i} - v_{i'}) \ \ \mbox{for} \ \ 1\leq  i < i' < \lambda$.
When for any $\gamma\in \Gamma$ it holds $\val(v_{i} - v_{i'})>\gamma$ for  $ i < i' < \lambda$
large, that is for $i'>i> \nu$ for some $\nu<\lambda$ we call $v$ {\em fundamental}. In this case a pseudo limit of $v$ is called {\emph limit}.

 We say that $v$  is 
\begin{enumerate}
\item
\emph{algebraic} if some $f \in V[T]$ satisfies $\val(f(v_{i})) < \val(f(v_{i'}))$ for large enough $1\leq i < i' < \lambda$;

\item
\emph{transcendental} if each $f \in V[T]$ satisfies $\val(f(v_{i})) = \val(f(v_{i'}))$ for large enough $1\leq i < i' < \lambda$.
\end{enumerate}

 The following  lemma is a variant  of  Ostrowski's lemma (\cite[ page 371, IV and III]{O}, see especially \cite[(II,4), Lemma 8]{Sch}).

\begin{Lemma}(Ostrowski) \label{o} Let $\beta_1,\ldots,\beta_m$, $m>1$ be any elements of an ordered abelian group $G$, $\lambda$ a limit ordinal and
 let $\{\gamma_s\}_{1\leq s<\lambda}$ be an increasing sequence of elements of G. Let $ t_1,\ldots, t_m$, be distinct integers.
 Then the following statements hold
 
 \begin{enumerate}
 
\item  
There exist an ordinal $1\leq \nu<\lambda$ and an integer $1\leq r\leq m$ such that
$$\beta_i+t_i\gamma_s>\beta_r+t_r\gamma_s$$ 
for all $i\not = r$ and $s>\nu$. 
 
\item There  exists
an ordinal $1\leq \nu<\lambda$ such that $(\beta_i+t_i\gamma_s)_{1\leq i\leq m}$ are different for all $s>\nu$. 

\item If the  integers   $ t_1,\ldots, t_m$, are not  distinct but
 $\beta_i\not =\beta_j$, $i\not =j$ whenever $t_i=t_j$ then  there  exists
an ordinal $1\leq \nu<\lambda$ such that $(\beta_i+t_i\gamma_s)_{1\leq i\leq m}$ are different for all $s>\nu$.
\end{enumerate}
\end{Lemma}

\begin{proof} The proofs from  the quoted papers gives the first statement even when the integers  $t_i$ are not necessarily positive.  The second statement follows
 applying   iteratively the first one.
 
 We need only (2), (3) in our paper and so we give a  proof only for  these two. Assume that $t_1<\ldots <t_m$.  If $m=2$ then there exists $1\leq \nu<\lambda$ such that $(t_2-t_1)\gamma_s\not =\beta_1-\beta_2$  for all $s>\nu$, which ends the case $m=2$. Apply induction on $m$. Let $m>2$.  By induction hypothesis there exists $1\leq\nu_1<\lambda$ such that $(\beta_i+t_i\gamma_s)_{1\leq i< m}$ are different for all $s>\nu_1$.  On the other hand, there exists $1\leq \nu_2<\lambda$ such that  $(t_m-t_i)\gamma_s\not =\beta_i-\beta_m$ for all $s>\nu_2$ and $1\leq i<m$. Then $\nu=\max\{\nu_1,\nu_2\}$ works for (2).

For (3) apply induction on $m$. Let $m>2$, the case $m=2$ being trivial.
 Assume that if $t_i=t_j $ for some $i\not =j$ then $\beta_i\not =\beta_j$. By induction hypothesis we  assume that there exists $1\leq \nu_1<\lambda$ such that $(\beta_e+t_e\gamma_s)_{1\leq e<m}$ are different, if $s>\nu_1$. On the other hand, we see that there exists $1\leq \nu_2<\lambda$ such that $\beta_m+t_m\gamma_s\not =\beta_j+t_j\gamma_s $ for all $j\not = m$ and $s>\nu_2$ because  $\beta_m\not =\beta_j$ if $t_m=t_j$ and so $(\beta_e+t_e\gamma_s)_{1\leq e\leq m}$ are different for  $s>\nu=\max \{\nu_1,\nu_2\}$.
\hfill\  \end{proof}

Next we give some small  extensions of  Lemma \ref{o} for the case of two  increasing sequences. 

\begin{Lemma} \label{o1} Let $\beta_0,\beta_1$, $c$ be some  elements of an ordered abelian group $G$, and let $(\gamma_{0,j_0})_{1\leq j_0<\lambda}$  be an increasing sequence of elements $\geq 0$ of $G$ and let  $(\gamma_{1,j_0})_{1\leq j_0<\lambda}$ be defined by $\gamma_{1,j_0}=\gamma_{0,j_0}$+c for $1\leq j_0<\lambda$. 
 Then there  exist a subset $A\subset [1,\lambda)$ and a map $\sigma:A\to [1,\lambda)$ such that  
the elements 

\begin{enumerate}
\item
$$P_{0,j_0}=\beta_0+\gamma_{0,j_0},$$
  
  \item 
  $$P_{1,j_1}=\beta_1+\gamma_{1,j_1}$$
   \end{enumerate}
  are different  for all  $1\leq j_0, j_1<\lambda$ with $j_1\not =\sigma(j_0)$ if $j_0\in A$. 
\end{Lemma}
\begin{proof}  If  $P_{0,j_0}\not =P_{1,j_1}$ for all $1\leq j_0,j_1<\lambda $ then take $A=\emptyset$ and there exists nothing to show. Otherwise, let $A$ be the set of all $1\leq j_0<\lambda$ such that there exists $1\leq j_1<\lambda$ with   $P_{0,j_0}=P_{1,j_1}$ (we may have $A=[1,\lambda)$). Let $j_0\in A$ and $j_1$ be such that
 $P_{0,j_0}=P_{1,j_1}$.   If $P_{0,j_0'}=P_{1,j_1'}$ for  some $(j_0',j_1')\not =(j_0,j_1)$ then we get 
$\gamma_{0,j_0'}-\gamma_{0,j_0}=\gamma_{1,j_1'}-\gamma_{1,j_1}=\gamma_{0,j_1'}-\gamma_{0,j_1}$. Clearly, for fix $j_0'$  there exists  at most one such  $j_1'$ because  
 $(\gamma_{0,j_0})_{1\leq j_0<\lambda}$ is increasing. Thus we can define $\sigma(j_0)=j_1$, $\sigma(j_0')=j_1'$. Thus $A,\sigma$ works.
\hfill\ \end{proof} 

\begin{Lemma} \label{o2} Let $\beta_0,\beta_1,\beta_{0,1}$ be some  elements of an ordered abelian group $G$, and let  $(\gamma_{0,j_0})_{1\leq j_0<\lambda_0}$, $(\gamma_{1,j_1})_{1\leq j_1<\lambda_1}$  be two  increasing sequences of elements $\geq 0$ of $G$. 
 Then there  exist two ordinals $1\leq \rho_0<\lambda_0$, $1\leq \rho_1<\lambda_1$, a subset $A\subset [1,\lambda_0)$ and a map $\sigma:A\to [1,\lambda_1)$ such that  
the elements 

\begin{enumerate}
\item
$$P_{0,j_0}=\beta_0+\gamma_{0,j_0},$$
  
  \item 
  $$P_{1,j_1}=\beta_1+\gamma_{1,j_1}$$
  
  \item
  $$P_{0,1,j_0,j_1}=\beta_{0,1}+\gamma_{0,j_0}+\gamma_{1,j_1}$$
   \end{enumerate}
  are different  for all  $\rho_0<j_0<\lambda_0$, $\rho_1<j_1<\lambda_1$ with $j_1\not =\sigma(j_0)$ if $j_0\in A$. In particular, for all  two ordinals $1\leq \nu_0<\lambda_0$, $1\leq \nu_1<\lambda_1$, there exist  $\nu_0<j_0<\lambda_0$, $ \nu_1<j_1<\lambda_1$ such that $P_{0,j_0}$, $P_{1,j_1}$, $P_{0,1,j_0,j_1}$  are different.
\end{Lemma}
\begin{proof} As above, let $A$ be the set of all $1\leq j_0<\lambda_0$ such that there exists $1\leq j_1<\lambda_1$ with   $P_{0,j_0}=P_{1,j_1}$. Let $j_0\in A$ and $j_1$ be such that
 $P_{0,j_0}=P_{1,j_1}$.   If $P_{0,j_0'}=P_{1,j_1'}$ for  some $(j_0',j_1')\not =(j_0,j_1)$ then we get 
$\gamma_{0,j_0'}-\gamma_{0,j_0}=\gamma_{1,j_1'}-\gamma_{1,j_1}$. Clearly, for fix $j_0'$  there exists  at most one such  $j_1'$ because  
 $(\gamma_{1,j_1})_{j_1<\lambda_1}$ is  increasing. Define $\sigma(j_0)=j_1$, $\sigma(j_0')=j_1'$.

Now, assume that   $P_{0,j_0}=P_{0,1,j_0,j_1}$. Then $\gamma_{1,j_1}=\beta_0-\beta_{0,1}$. As $(\gamma_{1,j_1})$
is  increasing we see that $\gamma_{1,j_1'}\not=\beta_0-\beta_{0,1}$ for $j'_1>j_1$, so  for $j'_1$ large. In this way, we find $\nu_0,\nu_1$, which work together with
  $A,\sigma$.
\hfill\ \end{proof} 

As an application of Lemma \ref{o2} we get
 the following lemma.

\begin{Lemma} \label{ka0}
Let $V \subset V'$ be an  immediate extension of valuation rings, $K$, $K'$ the fraction fields of $V$, $V'$, $\mm, \mm'$ the maximal ideals of $V$, $V'$, $y_0\in V'$ and $y_1=y_0^2$. Assume that $K'=K(y_0)$ and  $y_e$, $0\leq e\leq 1$ is  a pseudo limit of a pseudo convergent sequence $v_e=(v_{e,j_e})_{1\leq j_e<\lambda_e}$ over $V$, which  has no pseudo limits in $K$.
 Set $ y_{e,j_e}=( y_e- v_{e,j_e})/( v_{e,j_e+1}- v_{e,j_e})$, $0\leq e\leq 1$.
Then for every polynomial   $g\in V[Y_0,Y_1]$ with $\deg_{Y_i} g\leq 1$, $i=0,1$ and every ordinals $1\leq \nu_0<\lambda_0$, $1\leq \nu_1<\lambda_1$  there exist some  $\nu_0<j_0<\lambda_0$, $\nu_1<j_1<\lambda_1$, $c\in V\setminus \{0\}$ and a polynomial $g_1\in V[Y_{0,j_0}, Y_{1,j_1}]$ such that 
$$g(y_0,y_1)=cg_1(y_{0,j_0}, y_{1,j_1})$$
with $g_1\not \in \mm V[Y_{0,j_0}, Y_{1,j_1}]$ and 
the values of the coefficients of  all monomials  of $g_1-g_1(0,0)$ are different.
\end{Lemma}

\begin{proof} Let $g=b_0Y_0+b_1Y_1+b_2Y_0Y_1+b_3\in V[Y_0,Y_1]$ for $b_0,b_1,b_2,b_3\in V$ and $1\leq j_0<\lambda_0$, $1\leq j_1<\lambda_1$ and fix some ordinals $1\leq \nu_0<\lambda_0$, $1\leq \nu_1<\lambda_1$. 
By Taylor expansion  we have 
$g(y_0,y_1)=$
$$g( v_{0,j_0}, v_{1,j_1})+(b_0+b_2v_{1,j_1})(y_0- v_{0,j_0})+(b_1+b_2v_{0,j_0})(y_1- v_{1,j_1})+b_2(y_0- v_{0,j_0})(y_1- v_{1,j_1})=$$
$$g( v_{0,j_0}, v_{1,j_1}))+(b_0+b_2v_{1,j_1})(v_{0,j_0+1}- v_{0,j_0})y_{0,j_0}+(b_1+b_2v_{0,j_0})(v_{1,j_1+1}- v_{1,j_1})y_{1,j_1}+$$
$$b_2(v_{0,j_0+1}- v_{0,j_0})(v_{1,j_1+1}- v_{1,j_1})y_{0,j_0}y_{1,j_1}.$$

Set $\gamma_0=\val((b_0+b_2v_{1,j_1})(v_{0,j_0+1}- v_{0,j_0}))$,  $\gamma_1=(b_1+b_2v_{0,j_0})(v_{1,j_1+1}- v_{1,j_1})$, $\gamma_2=\val(b_2(v_{0,j_0+1}- v_{0,j_0})(v_{1,j_1+1}- v_{1,j_1}))$. 
Note that $y_{0,j_0}$, $y_{1,j_1}$ are units in $V'$. Also see that  $\beta_0=\val(b_0+b_2v_{1,j_1})$, $\beta_1=\val(b_1+b_2v_{0,j_0})$  are constant for $j_0,j_1$ large (that is for $j_0,j_1>\nu$ for some $1\leq \nu<\lambda$). This is explained for example in the proof of \cite[Theorem 3]{Kap}.   Actually,  otherwise it follows that for example  $(\val(v_{1,j_1}))_{1\leq j_1<\lambda}$  is increasing and if it is not fundamental then $(v_{1,j_1})_{1\leq j_1<\lambda}$ has a pseudo limit in $K$, which is false (note in this case that  every $z\in V$ with $\val(z)> \val(v_{1,j_1})=\val(v_{1,j_1+1}-v_{1,j_1})$, $j_1<\lambda$ is a pseudo limit of $v_1$). If  $(v_{1,j_1})_{1\leq j_1<\lambda}$
is fundamental then its value is $\val(y_1)$ for $j_1$ large, that is constant for $j_1$ large.
 Essentially is that the total degree of each $\partial g/\partial Y_i$ is $\leq 1$, $0\leq i\leq 1$. By Lemma \ref{o2} applied for $\beta_0$, $\beta_1$, $\beta_{0,1}=\val(b_2)$   and $\gamma_{0,j_0}'=\val(v_{0,j_0+1}- v_{0,j_0})$, $\gamma_{1,j_1}'=\val(v_{1,j_1+1}- v_{1,j_1})$ we get that these $\gamma_i$ are different for some high enough $j_0>\nu_0$, $j_1>\nu_1$. Note that $\val(g(y_0,y_1))\not = \gamma_i$ for all $i=0,1,2$ because $\val(g(y_0,y_1))$ is constant but the others increase.
Thus we have either  
$$(*)\ \ \val(g(y_0,y_1))=\val(g(v_{0,j_0},v_{1,j_1}))<\gamma_i,$$
 for all $0\leq i\leq 2$ for some $j_0,j_1$,   or  $\val(g(v_{0,j_0},v_{1,j_1}))=\gamma_i$ for some $0\leq i\leq 2$ and $j_0,j_1$, which is minimum among all $(\gamma_i)$.
If $(*)$ holds
  then for $c=g(v_{0,j_0},v_{1,j_1})$ we have $g(y_0,y_1)=cg_1(y_{0,j_0}, y_{1,j_1})$ for some polynomial $g_1\in  V[Y_{0,j_0}, Y_{1,j_1}]$ with  $g_1-1 \in \mm V[Y_{0,j_0}, Y_{1,j_1}])$. Moreover, the values of the coefficients of  all monomials  of $g_1-1$ are different. In general, $(*)$ does not hold as shows
  \cite[Example 16]{P3}.

Assume that $(*)$ does not hold,  $\gamma_1$ is the minimum between these $\gamma_e$ and so $\val(g(y_0,y_1))> \gamma_1$. Then $g(y_0,y_1)=c g_1(y_{0,j_0},y_{1,j_1})$ where $c=(b_1+b_2v_{0,j_0})(v_{1,j_1+1}- v_{1,j_1})$ and $g_1\in V[Y_{0,j_0},Y_{1,j_1}]$ is a polynomial with $\deg_{Y_e} g_1\leq 1$, $e=0,1$, the coefficient of $Y_{1,j_1}$ being $1$ and    the values of the coefficients of  all monomials  of $g_1-g_1(0,0)$ are different. Similarly, we treat the case when  $\gamma_0$ is the minimum between $\gamma_e$.

 Next assume that $(*)$ does not hold and $\gamma_2$ is the minimum between  $\gamma_e$. As above, for 
 $$c= b_2(v_{0,j_0+1}- v_{0,j_0})(v_{1,j_1+1}- v_{1,j_1})$$
  we get\ $g(y_0,y_1)=cg_1(y_{0,j_0},y_{1,j_1})$, where $g_1\in V[Y_{0,j_0},Y_{1,j_1}]$ is a polynomial with $\deg_{Y_e} g_1\leq 1$, $e=0,1$, the coefficient of $Y_{0,j_0}Y_{1,j_1}$ being $1$ and 
   the values of the coefficients of  all monomials  of $g_1-g_1(0,0)$ are different. 
 \hfill\ \end{proof} 

\vskip 0.3 cm
\section{Some extensions of Lemma \ref{ka0} for a polynomial of degree $\leq 1$ in several $(y_e)$.}
 The lemma bellow is an easy extension of Lemma \ref{o2}.

\begin{Lemma} \label{oe1} Let $m\in {\bf N}$, $S $ a set of nonempty subsets of $\{0,\ldots,m\}$, $(\beta_{\sigma})_{\sigma\in S}$ 
 some  elements of an ordered abelian group $G$, $(t_e)_{0\leq e\leq m}$ some positive integers  and let $(\gamma_{e,j_e})_{1\leq j_e<\lambda_e}$, $0\leq e\leq m$  be some increasing sequences of elements $\geq 0$ of $G$. For some ordinals $(j_e)_{0\leq e\leq m}$ with $j_e<\lambda_e$, $0\leq e\leq m$ and $\sigma\in S$ let $P_{\sigma,j_{\sigma}}=\beta_{\sigma}+
\sum_{e\in \sigma}t_e \gamma_{e,j_e}$,  where $j_{\sigma}=(j_e)_{e\in \sigma}$.
 Then for any ordinals $(\rho_e)_{0\leq e\leq m}$, $1\leq \rho_e<\lambda_e$ there exist some ordinals $(j_e)_{0\leq e\leq m}$ with $\rho_e<j_e<\lambda_e$ such that the elements
 $P_{\sigma,j_{\sigma}}$
 are different. 
\end{Lemma}

\begin{proof} We use induction on $m$, the case $m=0$ being trivial. Assume $m>0$ and choose some ordinals  $(\rho_e)_{0\leq e\leq m}$, $\rho_e<\lambda_e$. 
 Let $S'$ be the set of all $\sigma\in S$ which contains $m$ and set $\delta_{\sigma',j_{\sigma'}}=\beta_{\sigma'}+\sum_{e\in \sigma', e<m} t_e\gamma_{e,j_e}$ for $\sigma'\in S'$, $j_{\sigma'}=(j_e)_{e\in \sigma'}$.
By induction hypothesis there exist $\rho_0<j_0<\lambda_0,\ldots,\rho_{m-1}<j_{m-1}<\lambda_{m-1}$    such that $P_{\sigma,j_{\sigma}}$, $\delta_{\sigma',j_{\sigma'}}$ are different for any $\sigma\in S$ which does not contain $m$ and $\sigma'\in S'$. Fix these $j_e$, $e<m$. By Lemma \ref{o}  the elements 
$\delta_{\sigma',j_{\sigma'}}+t_m\gamma_{m,j_m}$ are all different  for $j_m$ large enough (note that all $\delta_{\sigma',j_{\sigma'}}$ are different), in particular for some $j_m>\nu_m$, where $\rho_m<\nu_m<\lambda_m$. Increasing $j_m$ (if necessary)  we may assume also that $P_{\sigma',j_{\sigma'}}$, $\sigma'\in S'$ are different from the fixed $P_{\sigma'',j_{\sigma''}}$,  $\sigma''\in S\setminus S'$ because $(\gamma_{m,j_m})$ is increasing.
\hfill\ \end{proof}

\begin{Lemma} \label{ka1}
Let $V \subset V'$ be an  immediate extension of valuation rings,  $K, K'$ the fraction fields of $V, V'$, $\mm, \mm'$ the maximal ideals of $V,V'$ and  $(y_e)_{0\leq e\leq m}$, $m\in {\bf N}$ some elements of $ V'$, which are not in $K$. Assume that for all $0\leq e\leq m$  the element $y_e$ is  a pseudo limit of a pseudo convergent sequence $v_e=(v_{e,j_e})_{1\leq j_e<\lambda_e}$ over $V$, which  has no pseudo limits in $K$.
 Set $y_{e,j_e}=(y_e-v_{e,j_e})/(v_{e,j_e+1}-v_{e,j_e})$.
Then for every polynomial   $g\in V[Y_0,\ldots,Y_m]$ such that   $ \deg_{Y_i} g\leq 1$ for all $0\leq i\leq m$ and $\partial g/\partial Y_i$,  $0\leq i\leq m$ has at most one monomial (that is, it has the form $a\Pi_eY_e$ for some $a\in V$ and several $e$)  and every ordinals $1\leq \nu_i<\lambda_i$, $0\leq i\leq m$  there exist   some  $\nu_0<j_0<\lambda_0,\ldots, \nu_m<j_m<\lambda_m$, $c\in V\setminus \{0\}$ and a polynomial $g_1\in V[Y_{0,j_0},\ldots, Y_{m,j_m}]$ such that 
$$g(y_0,\ldots,y_m)=cg_1(y_{0,j_0},\ldots, y_{m,j_m})$$
with $g_1\not\in \mm V[Y_{0,j_0},\ldots, Y_{m,j_m}]$
 and
 the values of the coefficients of  all monomials  of $g_1-g_1(0)$ are different. 
\end{Lemma}

\begin{proof} By Taylor expansion we have
$$g(y_0,\ldots,y_m)=\sum_{\tau\subset \{0,\ldots,m\}} (\partial^{|\tau|} g/\partial Y_{\tau}) (v_{0,j_0},\ldots,v_{m,j_m})\Pi_{e\in \tau} (y_e-v_{e,j_e}),$$
where $Y_{\tau}=(Y_e)_{e\in \tau}$ and $|\tau|$ denotes the cardinal of $\tau$. Since the value of a monomial in $(v_{e,j_e})$ is constant for $(j_e)$ large (that is constant for $j_0>\nu$ for some $1\leq \nu<\lambda$; see the proof of Lemma \ref{ka0})
we conclude  by our hypothesis  that $\beta_{\tau}=\val((\partial^{|\tau|} g/\partial Y_{\tau}) (v_{0,j_0},\ldots,v_{m,j_m}))$ is constant for $(j_e)$ large when $\tau\not =\emptyset$ as it was explained in the proof of Lemma \ref{ka0}.
 Set $\gamma_{e,j_e}'=\val(v_{e,j_e+1}- v_{e,j_e})=\val(y_e- v_{e,j_e})$.   By Lemma \ref{oe1} applied to $(\beta_{\tau})$, $(\gamma_{e,j_e}')$,  we see that 
$$\gamma_{\tau}=\val ((\partial^{|\tau|} g/\partial Y_{\tau}) (v_{0,j_0},\ldots,v_{m,j_m})\Pi_{e\in \tau} (y_e-v_{e,j_e}))$$
are different for all $\tau\not= \emptyset$ for some $(j_e)$, $j_e>\nu_e$. Also note that $\val(g(y_0,\ldots,y_m))$
 is different from $\gamma_{\tau}$, $\tau\not =\emptyset$ because the first one is constant but the others increase.
 It 
 follows that we have either 
$$(+)\ \ \val(g(y_0,\ldots, y_m))=\val(g(v_{0,j_0},\ldots,v_{m,j_m}))< \gamma_{\tau},$$
for some $(j_e)$, $j_e>\nu_e$ and  all $\tau\not =\emptyset$,
  or 
  $\val(g(v_{0,j_0},\ldots,v_{m,j_m}))=\gamma_{\sigma}$ for  some $(j_e)$, $j_e>\nu_e$ and some $\sigma\not = \emptyset$, where $\gamma_{\sigma}$ is the minimum among $(\gamma_{\tau})$, $\tau\not=\emptyset$.
If $(+)$ holds then  for $c=g(v_{0,j_0},\ldots,v_{m,j_m})$ we have $\val(c)<\gamma_{\tau}$ for all $\tau\not =\emptyset$. Thus
$g(y_0,\ldots, y_m)=c g_1(y_{0,j_0},\ldots, y_{m,j_m})$ for some polynomial $g_1\in  V[Y_{0,j_0},\ldots, Y_{m,j_m}]$ with  $g_1-1\in \mm V[Y_{0,j_0}, \ldots,Y_{m,j_m}]$. Moreover,   the values of the coefficients of  all monomials  of $g_1-1$ are different. 
   In  general,  $(+) $  does not hold for any $j_0,\ldots,j_m$ (see \cite[Example 16]{P3}).

Assume that $(+)$ does not hold,  $\gamma_{\sigma}$
 is the minimum between these $\gamma_{\tau}$, $\tau\not = \emptyset$  and so $\val(g(y_0,\ldots,y_m))>\gamma_{\sigma}$. Then $g(y_0,\ldots,y_m)=c g_1(y_{0,j_0},\ldots,y_{m,j_m})$   for  some $(j_e)$, $j_e>\nu_e$, where 
 $$c=(\partial^{|\sigma|} g/\partial Y_{\sigma}) (v_{0,j_0},\ldots,v_{m,j_m})\Pi_{e\in \sigma} (v_{e,j_e+1}-v_{e,j_e})$$
  and $g_1\in V[Y_{0,j_0},\ldots,Y_{m,j_m}]$ is a polynomial  of the form $\Pi_{e\in \sigma} y_{e,j_e}+g_1'$ for $g'_1\in  V[Y_{0,j_0},\ldots,Y_{m,j_m}]$ with  
  $$g_1'-g'_1(0)\in \mm V[Y_{0,j_0},\ldots,Y_{m,j_m}]),$$
  because $\gamma_{\sigma}$ is the minimum.  In fact $g_1(0)=g_1(v_{0,j_0},\ldots,v_{m,j_m})$ and
 the values of the coefficients of  all monomials  of $g_1-g_1(0)$ are different. 
\hfill\ \end{proof}

\begin{Lemma} \label{kat}
Let $V \subset V'$ be an  immediate extension of valuation rings,  $K, K'$ the fraction fields of $V, V'$, $\mm,\mm'$ the maximal ideals of $V,V'$ and  $(y_e)_{0\leq e\leq m}$, $m\in {\bf N}$ some elements of $ V'$, which are not in $K$. Assume that for all $0\leq e\leq m$  the element $y_e$ is  a pseudo limit of a pseudo convergent sequence $v_e=(v_{e,j_e})_{1\leq j_e<\lambda_e}$ over $V$, which  has no pseudo limits in $K$.
 Set $y_{e,j_e}=(y_e-v_{e,j_e})/(v_{e,j_e+1}-v_{e,j_e})$.
Let  $1\leq j_i<\lambda_i$, $0\leq i\leq m$ be some ordinals and  $g\in V[Y_{0,j_0},\ldots,Y_{m,j_m}]$  a polynomial such that   $ \deg_{Y_{e,j_e}} g\leq 1$ for all $0\leq i\leq m$. Assume that the values of the coefficients of  all monomials  of $g-g(0)$ are different and let  $j_i<t_i<\lambda_i$, $0\leq i\leq m$ be some ordinals. 
 Then there exist some ordinals $t_i\leq t'_i<\lambda_i$, $0\leq i\leq m$ and   a polynomial $g_1\in  V[Y_{0,t'_0},\ldots,Y_{m,t'_m}]$ such that 
$g(y_{0,j_0},\ldots, y_{m,j_m})=g_1(y_{0,t'_0},\ldots, y_{m,t'_m})$
 and  $g_1=g_1(0)+c_r  Y_{r,t'_r}+g_2,$
for some  $r\in \{0,\ldots,m\}$, $c_r\in V\setminus \{0\}$ and $g_2\in  c_r\mm V[Y_{0,t'_0},\ldots, Y_{m,t'_m}]$.   
\end{Lemma}

\begin{proof} 
 Changing  $Y_{e,j_e}$ by $ Y_{e,t_e}$  we see that 
$y_e-v_{e,j_e}=(v_{e,j_e+1}-v_{e,j_e}) y_{e,j_e}$ and $y_e-v_{e,t_e}=(v_{e,t_e+1}-v_{e,t_e}) y_{e,t_e}$.
Note that 
$\val(v_{e,j_e}-v_{e,t_e})=\val(v_{e,j_e+1}-v_{e,j_e})<\val(v_{e,t_e+1}-v_{e,t_e})$ and so
 $y_{e,j_e}=d_e+b_e y_{e,t_e}$ for an unit $d_e\in V$ and some $b_e\in \mm$. in fact $b_e=(v_{e,t_e+1}-v_{e,t_e})/(v_{e,j_e+1}-v_{e,j_e})$.
 Thus changing $j_e $ by $t_e$ means to  replace  $Y_{e,j_e}$ by $d_e+b_e Y_{e,t_e}$ in $ g$ obtaining a polynomial $g_1$.

Let $c\Pi_{i=1}^s Y_{e_i,j_{e_i}}$, $s>0$ be a monomial of $ g$. By the above transformation we get from that monomial a polynomial 
$$P=c\Pi_{i=1}^s(d_{e_i}+b_{e_i} Y_{e_i,t_{e_i}})\in V[Y_{0,t_0},\ldots, Y_{m,t_m}],$$
 where  the values of the coefficients of all monomials of $P$ 
  of degree $\geq 2$  are strictly bigger than the minimal  value of  coefficients of monomials of the linear part of $P-P(0)$ (that is of monomials of degree one), because in degree $\geq 2$  the coefficients involve  products  of at least two elements $b_e$.

 Let $h$ be the linear part of $ g_1$, that is  the sum of all monomials of degree $1$ of $ g_1$ (in fact the variables $Y_{e,t_e}$) with their coefficients from $g_1$.  Let $g=\sum_{\tau\subset \{0,\ldots,m\}} c_{\tau} Y_{\tau},$
where $Y_{\tau}=\Pi_{e\in \tau} Y_{e,j_e}$  and $c_{\tau}\in V$.
Then
$ g_1=\sum_{\tau\subset \{0,\ldots,m\}} c_{\tau}\Pi_{e\in \tau}(d_{e_i}+b_{e_i} Y_{e_i,t_{e_i}})$ 
and $h$ has the form 
$$h=\sum_{\tau\subset \{0,\ldots,m\}, \tau\not = \emptyset} c_{\tau} \sum_{e\in \tau}(\Pi_{e'\in (\tau\setminus \{e\})} d_{e'})b_e Y_{e,t_e}. $$
 The coefficient of $Y_{e,t_e}$  in $h$ is
$$q_e=\sum_{\tau\subset \{0,\ldots,m\}, e\in \tau} c_{\tau} (\Pi_{e'\in (\tau\setminus \{e\})}d_{e'})
b_e.$$
Let $c_e$ be the coefficient of minimal value from $(c_{\tau})_{\tau\subset \{0,\ldots,m\}, e\in \tau}$, which are different. Then $\val(q_e)=\val(c_e)+\val(b_e) $ and so the values of all coefficients of monomials of $g_1-g_1(0)$ 
  of degree $\geq 2$  are strictly bigger than some $\val(q_e)$, that is the value of a coefficient of a variable of $h$. We could have $\val(q_e)=\val(q_{e'})$, that is $\val(c_e)+\val(b_e)=\val(c_{e'})+\val(b_{e'})$  for some $e\not =e'$. Increasing a little $t_e$ to some $t'_e$ we see that $\val(b_e)$ increases. So we may find $(t'_e)$, $t'_e\geq t_e$ such that all $\val(c_e)+\val(b_e)$ are different, since $(\val(c_e))$ are different and constant and we may apply Lemma \ref{oe1}. Let us assume that the minimum is reached on the coefficient $q_r$ of $Y_{r,t'_r}$, which is enough.    
\hfill\ \end{proof}

\begin{Lemma} \label{ka2}
Let $V \subset V'$ be an  immediate extension of valuation rings,  $K, K'$ the fraction fields of $V, V'$, $\mm, \mm'$ the maximal ideals of $V,V'$ and  $(y_e)_{0\leq e\leq m}$, $m\in {\bf N}$ some elements of $ V'$, which are not in $K$. Assume that for all $0\leq e\leq m$  the element $y_e$ is  a pseudo limit of a pseudo convergent sequence $v_e=(v_{e,j_e})_{1\leq j_e<\lambda_e}$ over $V$, which  has no pseudo limits in $K$.
 Set $y_{e,j_e}=(y_e-v_{e,j_e})/(v_{e,j_e+1}-v_{e,j_e})$.
Then for every polynomial   $g\in V[Y_0,\ldots,Y_m]$ such that   $ \deg_{Y_i} g\leq 1$ for all $0\leq i\leq m$  and every ordinals $1\leq \nu_i<\lambda_i$, $0\leq i\leq m$  there exist some  $\nu_0<j_0<\lambda_0,\ldots, \nu_m<j_m<\lambda_m$ and a polynomial $g_1\in V[Y_{0,j_0},\ldots, Y_{m,j_m}]$ such that 
$$g(y_0,\ldots,y_m)=g_1(y_{0,j_0},\ldots, y_{m,j_m})$$
 and  $g_1=g_1(0)+c_r  Y_{r,j_r}+g_2,$
for some  $r\in \{0,\ldots,m\}$, $c_r\in V\setminus \{0\}$ and $g_2\in  c_r\mm V[Y_{0,j_0},\ldots, Y_{m,j_m}]$.   
\end{Lemma}
\begin{proof} 
 Fix $g$. Then $g$ has the form
$$g=\sum_{\tau\subset \{0,\ldots,m\}}a_{\tau }\Pi_{e\in \tau}Y_e,$$
$a_{\tau}\in V$.
Let $S_e$ be the set of all $\tau\subset  \{0,\ldots,m\}$ containing $e$, $0\leq e\leq m$. Introduce some new variables $(Y_{e,\tau})$, $0\leq e\leq m$, $\tau\in S_e$ and let $\rho:V[(Y_{e,\tau})]\to V[Y_0,\ldots,Y_m]$ be the map given by $Y_{e,\tau}\to Y_e$. Define
$$G=\sum_{\tau\subset \{0,\ldots,m\}}a_{\tau }\Pi_{e\in \tau}Y_{e,\tau}.$$
We have $\rho(G)=g$. Note that $y_{e,\tau}=y_e$ is a pseudo limit of $v_{e,\tau}=v_e$. Set  $y_{e,\tau,j_e}= (y_{e,\tau}-v_{e,\tau,j_e})/(v_{e,\tau,j_e+1}-v_{e,\tau,j_e})$.

Clearly,    $\partial G/\partial Y_{e,\tau}$ has at most one monomial for all $0\leq e\leq m$, $\tau\in S_e$ and we may apply Lemma \ref{ka1}. Thus for some  
ordinals $\nu_i<\lambda_i$, $0\leq i\leq m$  there exist  some ordinals $\nu_0<j_{0,\tau_0}<\lambda_0,$ $\tau_0\in S_0$, \ldots, $\nu_m<j_{m,\tau_m}<\lambda_m$, $\tau_m\in S_m$ and a polynomial $G_1\in V[(Y_{e,\tau_e,j_{e,\tau_e}})_{e,\tau_e}]$ such that 
$$G((y_{e,\tau_e})_{e,\tau_e})=G_1((y_{e,\tau_e,j_{e,\tau_e}})_{e,\tau_e})$$
 and the values of the coefficients of all monomials of $G_1-G_1(0)$ are different. Certainly, we may have  $j_{e,\tau_e}\not = j_{e,\tau'_e}$ for some $\tau_e\not = \tau'_e$ from $S_e$.

  Now  choose some ordinals $t_e<\lambda_e$ with $t_e>j_{e,\tau_e}$ for $0\leq e\leq m$, $\tau_e\in S_e$. By Lemma \ref{kat}
  there exist some ordinals $t_e\leq t'_e<\lambda_e$ and a polynomial $G_2\in  V[(Y_{e,\tau_e,t'_e})_{0\leq e\leq m, \tau_e\in S_e}]$ such that 
$$G_1((y_{e,\tau_e,j_{e,\tau_e}})_{e,\tau_e})=G_2((y_{e,\tau_e,t'_e})_{0\leq e\leq m,\tau_e\in S_e})$$  
 and  $G_2=G_2(0)+c_{r,\tau_r}  Y_{r,\tau_r,t'_r}+G_3,$
for some  $r\in \{0,\ldots,m\}$, $\tau_r\in S_r$, $c_{r,\tau_r}\in V\setminus \{0\}$ and $G_3\in  c_{r,\tau_r}\mm  V[(Y_{e,\tau_e,t'_e})_{0\leq e\leq m, \tau_e\in S_e}]$.

 Applying to $G_2$ the map $\rho_1:V[(Y_{e,\tau_e,t'_e})]\to V[(Y_{e,t'_e})]$ induced by $\rho$  we see that \\
   $g_1= \rho_1(G_2)$ satisfies
$$g(y_0,\ldots,y_m)=G_2((y_{e,\tau_e,t'_e})_{0\leq e\leq m, \tau_e\in S_e})=g_1(y_{0,t'_0},\ldots, y_{m,t'_m})$$
 and  $g_1=g_1(0)+c_r  Y_{r,t'_r}+g_2,$
 for $c_r=\sum_{\tau_r\in S_r} c_{r,\tau_r}$ and some $g_2\in  c_r\mm V[Y_{0,t'_0},\ldots, Y_{m,t'_m}]$. Note that $\val(c_r)$ is $\val(c_{r,\tau_r})$, which is the minimum among $\val(c_{r,\tau'_r})$, $\tau'_r\in S_r$.
 \hfill\ \end{proof}

 \section{Pure transcendental, immediate extensions of valuation rings}

Using the proofs of Lemmas \ref{ka1}, \ref{kat}, \ref{ka2} we get the following proposition.

\begin{Proposition} \label{p}
Let $V \subset V'$ be an  immediate extension of valuation rings,  $K, K'$ the fraction fields of $V, V'$, $\mm$ the maximal ideal of $V$ and  $y\in K'$ an element  which is not in  $K$. Assume that  $y$ is  a pseudo limit of a pseudo convergent sequence $v=(v_j)_{1\leq j<\lambda}$ over $V$, which has no pseudo limit in $K$.
 Set $y_j=(y-v_j)/(v_{j+1}-v_j)$.
Then for every nonzero polynomial  $g =\sum_{e=0}^m a_e Y^e\in V[Y]$ such that  $a_e=0$ for any $e>o$ multiple of $p$ if char\ $K=p>0$,  and every ordinal $1\leq \nu<\lambda$,  there exist some  $\nu<t<\lambda$ and a polynomial $g_1\in V[Y_t]$ such that 
$g(y)=g_1(y_t)$
 and    $g_1=g_1(0)+c  Y_t+g_2,$
for some $c\in V\setminus \{0\}$ and  $g_2\in  c\mm Y_t^2 V[Y_t]$. 
\end{Proposition}
\begin{proof} 
 Fix $g$ and $\nu$.
 Introduce some new variables $(Y_{e})$, $1\leq e\leq m$ and 
  define
$$G=a_0+\sum_{e=1}^m a_e  Y_{e}^e.$$
 Let $\rho:V[(Y_{e'})]\to V[Y]$ be the map given by $Y_{e'}\to Y$. 
We have $\rho(G)=g$. Note that $y_{e}=y$ is a pseudo limit of $v_{e,j_{e}}=v_{j_{e}}$, $1\leq j_{e}<\lambda$. Set $y_{e,j_{e}}=(y_{e}-v_{e,j_{e}})/(v_{e,j_{e}+1}-v_{e,j_{e}})$.

Consider the Taylor expansion \footnote{The polynomials $D^{(n_1,\ldots,n_m)}f \in R[Y_1,\ldots,Y_m]$ for 
$f=\sum_{n_1,\ldots, n_m} a_{n_1,\ldots, n_m} Y_1^{n_1}\cdots Y_m^{n_m} \in R[Y_1,\ldots,Y_m]$ make sense for any ring $R$: indeed, one constructs the Taylor expansion in the universal case
$R = {\bf Z}[(a_{n_1,\ldots, n_m})]$ by using the equality $n_1!\cdots n_m! \cdot (D^{((n_1,\ldots,n_m)} f) = \partial f/\partial (Y_1^{n_1},\ldots, Y_m^{n_m})$.}
\[
\tst  G(y) = \sum_{e = 0}^{m}\sum_{n_e=0}^e (D^{(n_e)}G)(v_{e,j_e}) \cdot (y_e - v_{e,j_e})^{n_e}\qxq{with}  D^{(n_e)}G \in V[Y_e]
\]
given by $\partial G/\partial Y_e^{n_e}$. Clearly,    $D^{(n_e)}G$ has just one monomial
 and we  follow the proof of  Lemma \ref{ka1}.

Since the value of a monomial in $(v_{e,j_{e}})$ is constant for $j_{e}$ large
we conclude   that $\beta_{e,n_e}=\val((D^{(n_e)}G)(v_{e,j_e})) $, $1\leq e\leq m$,
$0\leq n_e\leq  e$ are constant for $j_{e}>\nu$ large. As $v_{e,j_e}=v_{j_e}$ we see that $j_e$ does not depend really on $e$ but this will not affect us. Note that $\partial G/\partial (Y_{e}, Y_{e'})=0$ when $e\not =e'$. Set $\gamma_{e,j_{e}}'=\val(v_{e,j_{e}+1}- v_{e,j_{e}})=\val(y_{e}- v_{e,j_{e}})$ (in fact this is $\val(y-v_{j_{e}}))$.  

 By Lemma \ref{oe1} applied to $(\beta_{n_e})$, $(\gamma_{e,j_{e}}')$,  we see that 
$$ \gamma_{e,n_e}=\val ( (D^{(n_e)}G)(v_{e,j_e}) \cdot (v_{e,j_e+1} - v_{e,j_e})^{n_e})$$
 are different for all  $1\leq e\leq m$, $0\leq n_e\leq e$ with $a_{e}\not = 0$, for some $(j_{e})$, $j_e>\nu$.

Changing from $ Y_{e}$ to $ Y_{e,j_{e}}$ by $Y_e\to ((v_{e,j_e+1}-v_{e,j_e})Y_{e,j_e}+v_{e,j_e})$ we get a polynomial $G_1\in V[ (Y_{e,j_{e}})]$ 
$$G_1=\sum_{e = 0}^{m}\sum_{n_e=0}^e (D^{(n_e)}G)(v_{e,j_e}) \cdot (v_{e,j_e+1} - v_{e,j_e})^{n_e}Y_{e,j_e}^{n_e}$$
such that 
 $G( (Y_{e}))=G_1((Y_{e,j_{e}}))$ and 
the values of the nonzero coefficients of  all monomials  of $G_1-G_1(0)$ are different. 

  Assume that $G_1-G_1(0)= \sum_{e=1}^m \sum_{r=1}^e q_{e,r} Y_{e,j_e}^r$ for some $q_{e,r}\in V$ of different values. In characteristic zero we see that $q_{e,1}\not =0$ for all $e$. If char\ $K=p>0$ and $a_e=0$ for all $e$ multiple of $p$ as it is assumed,  we  have again $q_{e',1}\not =0$ for all $e'$. Note that our assumption does not imply that in $G_1$ there are no monomials of degree multiples of $p$, but this does not affect us.

Now we follow the proof of Lemma \ref{kat}.
We change  $Y_{e,j_{e}}$ by $ Y_{e,j'_{e}}$  for some $j_{e}<j'_{e}<\lambda$ for all $e$. We see that this means to  replace  $Y_{e,j_{e}}$ by $d_{e,j'_{e}}+
b_{e,j'_{e}} Y_{e,j'_{e}}$ for some units $d_{e,j'_{e}}\in V$ and some $b_{e,j'_{e}}\in \mm$ in $ G_1$ obtaining a polynomial $G_2\in V[(Y_{e,j'_{e}})_{e})]$ such that $G_1((y_{e,j_{e}}))=G_2((y_{e,j'_{e}}))$. 

Since the values  of the coefficients of  all monomials  of $G_1-G_1(0)$ are different we see as in the proof  of Lemma \ref{kat} that  the values of the coefficients of all monomials $Y_{e,j_e}^r$, $r>1$ of $G_2$ 
   are strictly bigger than the minimal  value of the coefficients of $Y_{e,j_e}$. Indeed, a monomial  
$c\Pi_{e=1}^s Y_{e,j_e}$, $s>0$  of $ G_1$ goes by the above transformation   in the polynomial 
$P=c\Pi_{e=1}^s(d_{e,j'_e}+b_{e,j'_e} Y_{e,j'_e})$,
 where      the coefficients of monomials of degree $>1$ involve   products  of at least two elements $b_{e,j'_{e}}$. We remind that $q_{e,1}\not = 0$ for all $e$.
 
  Actually, we can choose  $j'_e$ such that we have still different the values of all coefficients of monomials of $G_2-G_2(0)$ using again Lemma \ref{oe1} (we can change directly from  $Y_e$ to  $Y_{e,j'_e}$). Changing $(j_e)$ by $(j'_e)$ we may assume that  the values  of the coefficients of  all monomials  of $G_1-G_1(0)$ are different and $G_1-G_1(0)=\sum_{e=1}^m c_e Y_{e,j_e}+G'_1$ for some $c_e\in V\setminus \{0\}$ and 
	$G'_1\in  (c_1,\ldots,c_m)\mm ((Y_{e',j_{e'}}Y_{e'',j_{e''}}))V[(Y_{e',j_{e'}})]$.

Certainly, $(c_e)$ have different values but we would like to have $(c_e/(v_{e,j_e+1}-v_{e,j}))$ of different values. For this aim we will increase $(j'_e)$ to some $(j_e'')$ with the help of Lemma \ref{oe1} such that the elements $(c_e\Pi_{e'=1,e'\not =e}^m (v_{e',j''_{e'}+1}-v_{e',j''_{e'}}))$  have different values. Here it is important that $(c_e)$ have different values. Then $(c_e/(v_{e,j''_e+1}-v_{e,j''_e}))$ have different values too.

We change  $Y_{e',j_{e'}}$ by $Y_{e',j''_{e'}}$    for all $e'$. This means to  replace  $Y_{e',j_{e'}}$ by $d_{e',j''_{e'}}+
b_{e',j''_{e'}} Y_{e,j''_{e'}}$ for some units $d_{e',j''_{e'}}\in V$ and some $b_{e',j''_{e'}}\in \mm$ in $ G_1$ obtaining a polynomial $G_2\in V[(Y_{e',j''_{e'}})_{e'})]$ such that $G_1((y_{e',j_{e'}}))=G_2((y_{e',j''_{e'}}))$.
 The values of the coefficients of $G_2$ are possible not different anymore, but as we have seen the values of the monomials  $Y_{e,j_e''}^r$, $r>1$ of $G_2$   are strictly bigger than the minimal  value of the coefficients of $Y_{e,j''_e}$, $1\leq e\leq m$.

Now we change  $Y_{e',j''_{e'}}$ by $Y_{e',t}$   for some $t$ with $j''_{e'}<t<\lambda$ for all $e'$. This means to  replace  $Y_{e',j''_{e'}}$ by $d_{e',t}+
b_{e',t} Y_{e,t}$ for some units $d_{e',t}\in V$ and some $b_{e',t}\in \mm$ in $ G_2$ obtaining a polynomial $G_3\in V[(Y_{e',t})_{e'})]$ such that $G_2((y_{e',j''_{e'}}))=G_3((y_{e',t}))$.

 Actually 
$$b_{e,t}=(v_{e,t+1}-v_{e,t})/(v_{e,j''_{e}+1}-v_{e,j''_{e}})=
(v_{t+1}-v_{t})/(v_{e,j''_{e}+1}-v_{e,j''_{e}}).$$
Thus the coefficients of the variables $Y_{e,t}$ in $ G_3 $ have the form \\
$(c_e/(v_{e,j''_e+1}-v_{e,j''_e}))(v_{t+1}-v_{t})$ and so they have different values.
 Then $G_3-G_3(0)$ has the form  $c'_e Y_{e,t}+ G_3'$ for some $e$, $c'_e\in V\setminus\{0\}$ and $G_3'\in c'_e\mm V[(Y_{e,t})_{e})]$.
 
Applying to $G_3$ the map $\rho_1:V[(Y_{e,t})_{e}]\to V[(Y_{t})]$ induced by $\rho$  we see that
   $g_1= \rho_1(G_3)$ satisfies
$g(y)=G_3((y_{e,t})_{e})=g_1(y_{t})$
 and  
$$g_1=g_1(0)+c Y_t+g_2,$$
 for  some $c\in  V\setminus \{0\}$  of value $\val(c'_e)$ 
 and  some $g_2=\rho_1(G'_3)\in  c\mm Y_t^2 V[Y_t]$.  Note that $c$ is the sum of the  coefficients of some  $Y_{e',t}$, which have  different values. 
 \hfill\ \end{proof}

\begin{Corollary} \label{co}
Let $V \subset V'$ be an  immediate extension of valuation rings,  $K, K'$ the fraction fields of $V, V'$, $\mm$ the maximal ideal of $V$ and  $y\in K'$ an  element  which is not in  $K$. 
Assume that char\ $K=p>0$ and $y$ is  a pseudo limit of a pseudo convergent sequence $v=(v_j)_{1\leq j<\lambda}$ over $V$, which  has no pseudo limit in $K$.
 Set $y_j=(y-v_j)/(v_{j+1}-v_j)$.
Then for every nonzero  polynomial  $g\in V[Y]$  and every ordinal $1\leq \nu<\lambda$ one of the following statements holds.

\begin{enumerate}
\item  There exist some  $\nu<t<\lambda$ and a polynomial $g_1\in V[Y_t]$ such that 
$g(y)=g_1(y_t)$
 and    $g_1=g_1(0)+c  Y_t+g_2,$ 
for some $c\in V\setminus \{0\}$ and  $g_2\in  c\mm Y_t^2 V[Y_t]$.
\item 
There exist some  $\nu<t<\lambda$ and a polynomial $g_1\in V[Y_t]$ such that 
$y g(y)=g_1(y_t)$
 and    $g_1=g_1(0)+c  Y_t+g_2,$ 
for some $c\in V\setminus \{0\}$ and  $g_2\in  c\mm Y_t^2 V[Y_t]$.
\end{enumerate}
\end{Corollary}

\begin{proof} 

If $g =\sum_{e=0}^m a_e Y^e\in V[Y]$ such that  $a_e=0$ for any $e$ multiple of $p$ if char $K=p>0$ then  apply Proposition \ref{p} and we get (1). 
	Suppose that  $g=f_1+f_2$ with $f_2\in V[Y^p]$ and $f_1=\sum_{e>0, a_e=0\ \mbox{if\ e\ is\ a\ multiple\ of\ p}}^m a_e Y^e\in V[Y]$. If $f_1=0$ then $Yg=Yf_2$ has no monomials in $V[Y^p]$ and applying Proposition \ref{p} to $Yg$ we get (2). 
	
	Assume that $f_1\not =0$.
	By Proposition \ref{p} we see that changing from $Y$ to some $Y_j $, $j>\nu$ we get $h\in V[Y_j]$ such that $h(y_j)=g(y)$ and $h=h_1+h_2$ with $h_2\in V[Y_j^p]$ and $h_1-h_1(0)=c'_1Y_j+h'_1$ with $h'_1\in c'_1\mm Y_j^2 V[Y_j]$. If $\val(c'_1)<\val(h_2(y_j)-h_2(0))$ then we get again (1) for some $c_1\in V$ with $\val(c_1)=\val(c'_1)$. If $\val(c'_1)=\val(h_2(y_j)-h_2(0))$ then we reduce to the above inequality because we may change from $Y_j$ to $Y_{j'}$ 
 for some $j'>j$ (note that $h_2\in V[Y_j^p]$). Thus we obtain again  (1).
	
	Now assume that $\val(c_1)>\val(h_2(y_j)-h_2(0))$ and   this inequality remains if we change as above from $Y_j$ to $Y_{j'}$ for any $j<j'<\lambda$.
	 Then note that in $Y f_2$ the  coefficients of monomials from $V[Y_j^p]$  are zero and we may apply  Proposition \ref{p} for $Yf_2$. Changing from $ Y$ to $Y_t$ for some $j<t<\lambda$ we get a polynomial $H_2\in V[Y_t]$ from $Y f_2$  such that $yf_2(y)=H_2(y_t)$ and $H_2-H_2(0)=c_2'Y_t+H'_2$ with $H'_2\in c_2'\mm V[Y_t]$. Under this transformation $Y g$ goes in a polynomial $H\in V[Y_t]$ and $Y f_1$ goes in a polynomial $H_1$ with $\val(H_1-H_1(0))> \val(c_2')$ by our above assumption (certainly $H_1$ could contain also some monomials from $V[Y^p]$).  Thus $yg(y)=H(y_t)$ and 
 $H-H(0)=c_2Y_t+H'$
for some $c_2\in V$ of value $\val(c'_2)$ and $H'\in c_2\mm Y_t^2 V[Y_t]$, that  is (2) holds. 
\hfill\ \end{proof}

This corollary will be applied in Lemma \ref{1} which gives in fact a small part of the proof of Theorem
\ref{key}.
In this theorem we  need a proposition as above  for polynomials $g$ in two variables. This is done below but the proof is more difficult. Actually we believe that the above proposition holds  for polynomials $g$ in more variables.

\begin{Proposition} \label{p'}
Let $V \subset V'$ be an  immediate extension of valuation rings,  $K, K'$ the fraction fields of $V, V'$, $\mm$ the maximal ideal of $V$ and  $y_1,y_2\in K'$ two elements,  which are not in $K$. Assume that $y_i$, $i=1,2$ is  a pseudo limit of a pseudo convergent sequence $v_i=(v_{i,j})_{1\leq j<\lambda_i}$, $i=1,2$ over $V$, which   have no pseudo limit in $K$.
 Set $y_{ij}=(y_i-v_{i,j})/(v_{i,j+1}-v_{i,j})$.
Then for every nonzero polynomial   $g=\sum_{e_1,e_2=0}^m a_{e_1,e_2} Y_1^{e_1} Y_2^{e_2} \in V[Y_1,Y_2]$, such that $a_{e_1,e_2}=0$ if char\ $K=p>0$ and at least one of $e_1,e_2$ is a multiple of  $p$,  and every ordinals $1\leq \nu_i<\lambda_i$,  $i=1,2$  there exist some  $\nu_i<t_i<\lambda_i$, $i=1,2$ and a polynomial $g_1\in V[Y_{1,t_1},Y_{2,t_2}]$ such that 
$g(y_1,y_2)=g_1(y_{1,t_1},y_{2,t_2})$
 and  $g_1=g_1(0)+c_1  Y_{1,t_1}+c_2 Y_{2,t_2}+g_2,$
for some   $c_1,c_2\in V, $ at least one of them nonzero, and $g_2\in  (c_1,c_2)\mm (Y_{1,t_1},Y_{2,t_2})^2 V[Y_{1,t_1},Y_{2,t_2}]$.   
\end{Proposition}
\begin{proof} 
 Fix $g$ and  $\nu_1<\lambda_1$,  $\nu_2<\lambda_2$.
 Introduce some new variables $(Y_{1;e_1,e_2},Y_{2;e_1,e_2})$, $1\leq e_1\leq m$,  $1\leq e_2\leq m$ and 
  define
$$G=a_{0,0}+\sum_{e_1,e_2=0,e_1+e_2>0}^m a_{e_1,e_2}Y_{1;e_1,e_2}^{e_1}Y_{2;e_1,e_2}^{e_2}.$$
 Let $\rho:V[(Y_{1;e'_1,e'_2}),(Y_{2;e'_1,e'_2})]\to V[Y_1,Y_2]$ be the map given by $Y_{1;e'_1,e'_2}\to Y_1$,  $Y_{2;e'_1,e'_2}\to Y_2$.
We have $\rho(G)=g$. Note that $y_{1;e_1,e_2}=y_1$ is a pseudo limit of $v_{1;e_1,e_2,j_{1;e_1,e_2}}=v_{1;j_{1;e_1,e_2}}$, $1\leq j_{1;e_1,e_2}<\lambda_1$ and
 $y_{2;e_1,e_2}=y_2$ is a pseudo limit of $v_{2;e_1,e_2,j_{2;e_1,e_2}}=v_{2,j_{2;e_1,e_2}}$, $1\leq j_{2;e_1,e_2}<\lambda_2$. Set 
 $$y_{1;e_1,e_2,j_{1;e_1,e_2}}=(y_{1;e_1,e_2}-v_{1;e_1,e_2,j_{1;e_1,e_2}})/(v_{1;e_1,e_2,j_{1;e_1,e_2}+1}-v_{1;e_1,e_2,j_{1;e_1,e_2}})$$
  and similarly 
  $$y_{2;e_1,e_2,j_{2;e_1,e_2}}=(y_{2;e_1,e_2}-v_{2;e_1,e_2,j_{2;e_1,e_2}})/(v_{2;e_1,e_2,j_{2;e_1,e_2}+1}-v_{2;e_1,e_2,j_{2;e_1,e_2}}).$$

 We  follow the proof of  Lemma \ref{ka1} and certainly the proof of Proposition \ref{p}.
 Consider the Taylor expansion
$$ g(y_1,y_2) = \sum_{e_1 =0}^{m}\sum_{n_{e_1}=0}^{e_1} \sum_{n_{e_2}=0}^{e_2}(D^{(n_{e_1},n_{e_2})}G)
(v_{1;e_1,e_2,j_{1;e_1,e_2}},v_{2;e_1,e_2,j_{2;e_1,e_2}}) \cdot $$
$$ (y_{1;e_1,e_2,j_{1;e_1,e_2}} - v_{1;e_1,e_2,j_{1;e_1,e_2}})^{n_{e_1}}(y_{2;e_1,e_2,j_{2;e_1,e_2}} - v_{2;e_1,e_2,j_{2;e_1,e_2}})^{n_{e_2}} 
$$
with  $ D^{(n_{e_1},n_{e_2})}G$ given by $\partial G/\partial (Y_{1;e_1,e_2,j_{1;e_1,e_2}}^{n_{e_1}}Y_{2;e_1,e_2,j_{2;e_1,e_2}}^{n_{e_2}})$. 
 Clearly,   $D^{(n_1,n_2)}G$ has just one monomial.
 
Since the value of a monomial in $(v_{1;e_1,e_2,j_{1;e_1,e_2}}), (v_{2;e_1,e_2,j_{2;e_1,e_2}})$ is constant for $j_{1;e_1,e_2}$, $j_{2;e_1,e_2} $ large,
we conclude   that 
$\beta_{e_1,e_2,n_{e_1},n_{e_2}}$ the value of \\ 
$(D^{(n_{e_1},n_{e_2})}G)((v_{1;e_1,e_2,j_{1;e_1,e_2}}), (v_{2;e_1,e_2,j_{2;em_1,e_2}}))$, $1\leq e_1,e_2\leq m$, $0\leq n_{e_1}\leq e_1$, $0\leq n_{e_2}\leq e_2$ are constant for $j_{1;e_1,e_2}>\nu_1$, $j_{2;e_1,e_2}>\nu_2$ large. 
  Note that \\
	$\partial G/\partial (Y_{1;e_1,e_2,}, Y_{1;e'_1,e'_2})=0$,
		$\partial G/\partial (Y_{2;e_1,e_2}, Y_{2;e'_1,e'_2})=0$, 
	$\partial G/\partial (Y_{1;e_1,e_2}, Y_{2;e'_1,e'_2})=0$ when $e_1\not =e'_1$ or $e_2\not =e'_2$.

 Set $\gamma_{1;e_1,e_2,j_{1;e_1,e_2}}'=\val(v_{1;e_1,e_2,j_{1;e_1,e_2}+1}- v_{1;e_1,e_2,j_{1;e_1,e_2}})=
 \val(y_{1;e_1,e_2}- v_{1;e_1,e_2,j_{1;e_1,e_2}})$ (in fact this is $\val(y_1-v_{1,j_{1;e_1,e_2}}))$
and $\gamma_{2;e_1,e_2,j_{2;e_1,e_2}}'=
\val(v_{2;e_1,e_2,j_{2;e_1,e_2}+1}- v_{2;e_1,e_2,j_{2;e_1,e_2}})=\val(y_{2;e_1,e_2}- v_{2;e_1,e_2,j_{2;e_1,e_2}})$.

    By Lemma \ref{oe1} applied to $\beta_{e_1.e_2,n_{e_1},n_{e_2}}$,  $\gamma_{1;e_1,e_2,j_{1;e_1e_2}}'$, $\gamma_{2;e_1,e_2,j_{2;e_1,e_2}}'$  we see that \\ 
$\gamma_{e_1,e_2,n_{e_1},n_{e_2}}=$
$$ \beta_{e_1,e_2,n_{e_1},n_{e_2}}+\val( (v_{1;e_1,e_2,j_{1;e_1,e_2}+1}-v_{1;e_1,e_2,j_{1;e_1,e_2}})^{n_{e_1}} (v_{2;e_1,e_2,j_{2;e_1,e_2}+1})-v_{2;e_1,e_2,j_{2;e_1,e_2}})^{n_{e_2}})$$
 are different for all  $1\leq e_1,e_2\leq m$, $0\leq n_{e_1}\leq e_1$, $0\leq n_{e_2}\leq e_2$
    for some $(j_{1;e'_1,e'_2})>\nu_1$, 
  $(j_{2;e'_1,e'_2})>\nu_2$ with
 $1\leq e'_1,e'_2\leq m$,  $a_{e'_1,e'_2}\not = 0$.

Changing from $ Y_{1;e_1,e_2}, Y_{2;e_1,e_2}$ to $ Y_{1;e_1,e_2,j_{1;e_1,e_2}},Y_{2;e_1,e_2,j_{2;e_1,e_2}} $ we get a polynomial $G_1\in V[ (Y_{1;e_1,e_2,j_{1;e_1,e_2}},Y_{2;e_1,e_2,j_{2;e_1,e_2}})]$ such that 
 $$G( (Y_{1;e_1,e_2}),(Y_{2;e_1,e_2}))=G_1((Y_{1;e_1,e_2,j_{1;e_1,e_2}}),(Y_{2;e_1,e_2,j_{2;e_1,e_2}}))$$
 and 
the values of the nonzero coefficients of  all monomials  of $G_1-G_1(0)$ are different. 

  Assume that 
	$$G_1-G_1(0)= \sum_{e_1,e_2=1}^m \sum_{r_1,r_2=0,r_1+r_2>0}^e q_{e_1,e_2,r_1,r_2} Y_{1;e_1,e_2,j_{1;e_1,e_2}}^{r_1} Y_{2;e_1,e_2,j_{2;e_1,e_2}}^{r_2}$$
	for some $q_{e_1,e_2,r_1,r_2}\in V$ of different value. In characteristic zero we see that $q_{e_1,e_2,1,0}\not =0$ and $q_{e_1,e_2,0,1}\not =0$ for all $e_1,e_2$. If char\ $K=p>0$ then we  have $q_{e_1,e_2,1,0} \not = 0$, $q_{e_1,e_2,0,1}\not =0$
	for all $e_1,e_2$ by our assumption  because both $e_1,e_2 $ are not multiples of $p$. Note that our assumption does not imply that in $G_1$ there are no monomials with degrees multiples of $p$ in $Y_{1;e_1,e_2,j_{1;e_1,e_2}}$, or $Y_{2;e_1,e_2,j_{2;e_1,e_2}}$, but this does not affect us.

Now we follow the proof of Lemma \ref{kat}.
We change  $Y_{1;e_1,e_2,j_{1;e_1,e_2}}$ by $Y_{1;e_1,e_2,j'_{1;e_1,e_2}}$ and   $Y_{2;e_1,e_2,j_{2;e_1,e_2}}$ by $Y_{2;e_1,e_2,j'_{2;e_1,e_2}}$ for some $j_{1;e_1,e_2}<j'_{1;e_1,e_2}<\lambda_1$ and  $j_{2;e_1,e_2}<j'_{2;e_1,e_2}<\lambda_2$ for all $e_1,e_2$. We see that this means to  replace    $Y_{1;e_1,e_2,j_{1;e_1,e_2}}$  by $d_{1;e_1,e_2,j'_{1;e_1,e_2}}+b_{1;e_1,e_2,
j'_{1;e_1,e_2}} Y_{1;e_1,e_2,j'_{1;e_1,e_2}}$ and  $Y_{2;e_1,e_2,j_{2;e_1,e_2}}$  by\\
 $d_{2;e_1,e_2,j'_{2;e_1,e_2}}+b_{2;e_1,e_2,j'_{2;e_1,e_2}} Y_{2;e_1,e_2,j'_{2;e_1,e_2}}$ 
 for some units $d_{1;e_1,e_2,j'_{1;e_1,e_2}},d_{2;e_1,e_2,j'_{2;e_1,e_2}}   \in V$ and some  $b_{1;e_1,e_2,j'_{1;e_1,e_2}}, b_{2;e_1,e_2,j'_{2;e_1,e_2}} \in \mm$ in $ G_1$ obtaining a polynomial \\
$G_2\in V[(Y_{1;e_1,e_2,j'_{1;e_1.e_2}})),(Y_{2;e_1,e_2,j'_{2;e_1,e_2}}))]$ such that
$$G_1((y_{1;e_1,e_2,j_{1;e_1.e_2}}), (y_{2;e_1,e_2,j_{2;e_1.e_2}}))=
G_2((y_{1;e_1,e_2,j'_{1;e_1.e_2}}), (y_{2;e_1,e_2,j'_{2;e_1.e_2}})).$$

Since the values  of the coefficients of  all monomials  of $G_1-G_1(0)$ are different we note using the proof  of Lemma \ref{kat} that  the values of the coefficients of all monomials of $G_2$ 
  of degree $\geq 2$  are strictly bigger than the minimal  value of the coefficients of monomials of the linear part  of $G_2-G_2(0)$, because in degree $\geq 2$  the coefficients involve  products  of at least two elements of
	$b_{1;e_1,e_2,j'_{1;e_1,e_2}}, b_{2;e_1,e_2,j'_{2;e_1,e_2}}.$ Here it is important that $q_{e_1,e_2,1,0} \not = 0$, $q_{e_1,e_2,0,1}\not =0$ for all $e_1,e_2$.

Actually, we can choose  $j'_{1;e_1,e_2}$,  $j'_{2;e_1,e_2}$ such that we have still different the values of all coefficients of monomials of $G_2-G_2(0)$ using again Lemma \ref{oe1}.  Changing $j_{1;e_1,e_2}$,  $j_{2;e_1,e_2}$ by $j'_{1;e_1,e_2}$,  $j'_{2;e_1,e_2}$ 
we may assume that  the values  of the coefficients of  all monomials  of $G_1-G_1(0)$ are different and for some $c_{1,e_1,e_2}, c_{2,e_1,e_2}\in V$ 
$$G_1-G_1(0)=\sum_{e_1,e_2=1}^m c_{1,e_1,e_2} Y_{1;e_1,e_2,j_{1;e_1,e_2}}+
\sum_{e_1,e_2=1}^m c_{2,e_1,e_2} Y_{2;e_1,e_2,j_{2;e_1,e_2}}+ G'_1$$ 
for  some
 $G'_1\in  ((c_{1,e_1,e_2}),(c_{2,e_1,e_2}))\mm  V[(Y_{1;e_1,e_2,j_{1;e_1.e_2}}),(Y_{2;e_1,e_2,j_{2;e_1,e_2}})]$.

Certainly, $((c_{1,e_1,e_2}),(c_{2,e_1,e_2})) $
 have different values but we would like to have $(c_{1,e_1,e_2} /(v_{1;e_1,e_2,j_{1;e_1,e_2}+1}-v_{1;e_1,e_2,j_{1;e_1,e_2}}))$  of different values and similarly\\
$(c_{2,e_1,e_2} /(v_{2;e_1,e_2,j_{2;e_1,e_2}+1}-v_{2;e_1,e_2,j_{2;e_1,e_2}}))$.

 For this aim we will increase $(j_{1;e_1,e_2})$, $(j_{2;e_1,e_2})$ to some  $(j''_{1;e_1,e_2})$, $(j''_{2;e_1,e_2})$  with the help of Lemma \ref{oe1} such that the elements 
$$c_{1;e_1,e_2}\Pi_{e_1'=1,e_1'\not =e_1}^m  \Pi_{e_2'=1}^m(v_{1;e_1',e_2',j''_{1;e_1',e_2'}+1}-v_{1;e_1',e_2',j''_{1;e_1',e_2'}}) (v_{2;e_1',e_2',j''_{2;e_1',e_2'}+1}-v_{2;e_1',e_2',j''_{2;e_1',e_2'}})$$ 
and
$$c_{2;e_1,e_2} \Pi_{e_1'=1}^m  \Pi_{e_2'=1,e_2'\not =e_2}(v_{1;e_1',e_2',j''_{1;e_1',e_2'}+1}-v_{1;e_1',e_2',j''_{1;e_1',e_2'}}) (v_{2;e_1',e_2',j''_{2;e_1',e_2'}+1}-v_{2;e_1',e_2',j''_{2;e'_1,e'_2}})$$ 
 have different values. Here, it is important that $ ( c_{1,e_1,e_2}), (c_{2,e_1,e_2})$ have different values. Then 
$$(c_{1,e_1,e_2} /(v_{1;e_1,e_2,j''_{1;e_1,e_2}+1}-v_{1;e_1,e_2,j''_{1;e_1,e_2}})), \ \
(c_{2,e_1,e_2} /(v_{2;e_1,e_2,j''_{2;e_1,e_2}+1}-v_{2;e_1,e_2,j''_{2;e_1,e_2}}))$$
 have different values too. 

We change  $Y_{1;e_1,e_2,j_{e_1,e_2}}$ by $Y_{1;e_1,e_2,j''_{1;e_1,e_2}}$ and   $Y_{2;e_1,e_2,j_{2;e_1,e_2}}$ by $Y_{2;e_1,e_2,j''_{2;e_1,e_2}}$. This means to  replace    $Y_{1;e_1,e_2,j_{1;e_1,e_2}}$  by $d_{1;e_1,e_2,j''_{1;e_1,e_2}}+b_{1;e_1,e_2,j''_{1;e_1,e_2}} Y_{1;e_1,e_2,j''_{1;e_1,e_2}}$ and \\ $Y_{2;e_1,e_2,j_{2;e_1,e_2}}$  by $d_{2;e_1,e_2,j''_{2;e_1,e_2}}+b_{2;e_1,e_2,j''_{2;e_1,e_2}} Y_{2;e_1,e_2,j''_{2;e_1,e_2}}$ 
 for some units\\
 $d_{1;e_1,e_2,j''_{1;e_1,e_2}},d_{2;e_1,e_2,j''_{2;e_1,e_2}}   \in V$ and some $b_{1;e_1,e_2,j''_{1;e_1,e_2}}, b_{2;e_1,e_2,j''_{2;e_1,e_2}} \in \mm$ in $ G_1$ obtaining a polynomial 
$G_2\in V[(Y_{1;e_1,e_2,j''_{1;e_1.e_2}})_{e_1,e_2}),(Y_{2;e_1,e_2,j''_{2;e_1,e_2}})_{e_1,e_2})]$ such that\\
$G_1((y_{1;e_1,e_2,j_{1;e_1.e_2}}), (y_{2;e_1,e_2,j_{2;e_1.e_2}}))=
G_2((y_{1;e_1,e_2,j''_{1;e_1.e_2}}), (y_{2;e_1,e_2,j''_{2;e_1.e_2}})).$ 

 The values of the coefficients of $G_2$ are possible not different anymore, but as we have seen the values of the monomials   of $G_2$ of degree  two  are strictly bigger than the minimal  value of the coefficients of the variables $((Y_{1;e_1,e_2,j''_{1;e_1.e_2}})_{e_1,e_2})$,\\
$((Y_{2;e_1,e_2,j''_{2;e_1,e_2}})_{e_1,e_2}) $.

Now change  $Y_{1;e_1,e_2,j''_{1;e_1,e_2}}$ by $Y_{1;e_1,e_2,t_1}$ and   $Y_{2;e_1,e_2,j''_{2;e_1,e_2}}$ by $Y_{2;e_1,e_2,t_2}$ for some\\
 $j''_{1;e_1,e_2}<t_1<\lambda_1$ and  $j''_{2;e_1,e_2}<t_2<\lambda_2$ for all $e_1,e_2$. This means to  replace    $Y_{1;e_1,e_2,j''_{1;e_1,e_2}}$  by $d_{1;e_1,e_2,t_1}+b_{1;e_1,e_2,t_1} Y_{1;e_1,e_2,t_1}$ and  $Y_{2;e_1,e_2,j''_{2;e_1,e_2}}$  by \\
$d_{2;e_1,e_2,t_2}+b_{2;e_1,e_2,t_2} Y_{2;e_1,e_2,t_2}$ 
 for some units 
$d_{1;e_1,e_2,t_1},d_{2;e_1,e_2,t_2}   \in V$ and some\\
 $b_{1;e_1,e_2,t_1}, b_{2;e_1,e_2,t_2} \in \mm$ in $ G_2$ obtaining a polynomial \\
$G_3\in V[(Y_{1;e_1,e_2,t_1})_{e_1,e_2}),(Y_{2;e_1,e_2,t_2})_{e_1,e_2})]$ such that\\
$G_2((y_{1;e_1,e_2,j''_{1;e_1.e_2}}), (y_{2;e_1,e_2,j''_{2;e_1.e_2}}))=
G_3((y_{1;e_1,e_2,t_1}), (y_{2;e_1,e_2,t_2})).$

  Actually 
$$b_{1;e_1,e_2,t_1}=(v_{1;e_1,e_2,t_1+1}-v_{1;e_1,e_2,t_1})/(v_{1;e_1,e_2,j''_{1;e_1,e_2}+1}-v_{1;e_1,e_2,j''_{1;e_1,e_2}})=$$
$$(v_{1,t_1+1}-v_{1,t_1})/(v_{1,j''_{1;e_1,e_2}+1}-v_{1,j''_{1;e_1,e_2}})$$
 and similarly  for $b_{2;e_1,e_2,t_2}$ as in the proof of Proposition \ref{p}.
Thus the coefficients of the variables $(Y_{1;e_1,e_2,t_1})$ in $ G_3 $ have the form \\
$(c_{1,e_1,e_2} /(v_{1;e_1,e_2,j''_{1;e_1,e_2}+1}-v_{1;e_1,e_2,j''_{1;e_1,e_2}}))$ multiplied with $(v_{1;e_1,e_2,t_1+1}-v_{1;e_1,e_2,t_1})$,  and  have different values. Similarly are the  coefficients of the variables $(Y_{2;e_1,e_2,t_2})$ in $ G_3 $.
 Then $G_3-G_3(0)$ has  for some $e_1,e_2$, $ c'_{1,e_1,e_2},c'_{2,e_1,e_2} \in V$ the form 
 $$c'_{1,e_1,e_2} Y_{1;e_1,e_2,t_1}+ c'_{2,e_1,e_2} Y_{2;e_1,e_2,t_2}+G_3'$$ for some  $G_3'\in (c'_{1,e_1,e_2},c'_{2,e_1,e_2})\mm V[(Y_{1;e_1,e_2,t_1}), (Y_{2;e_1,e_2,t_2})]$.

 Applying to $G_3$ the map 
$\rho_1:V[(Y_{1;e_1,e_2,t_1}),(Y_{2;e_1,e_2,t_2}) ]\to V[Y_{1,t_1},Y_{2,t_2}]$ induced by $\rho$  we see that
   $g_1= \rho_1(G_3)$ satisfies
$$g(y_1,y_2)=G_3((y_{1;e'_1,e'_2,t_1}),(y_{2;e'_1,e'_2,t_2}))=g_1(y_{1,t_1},y_{2,t_2})$$
 and  $g_1=g_1(0)+ c_1 Y_{1,t_1} +c_2 Y_{2,t_2}+g'_2$ 
for  some $c_1,c_2\in  V$, at least one of them nonzero,  of value $\val(c'_{1,e_1,e_2})$, respectively  $\val(c'_{2,e_1,e_2})$
 and  some\\
 $g'_2=\rho_1(G'_3)\in  (c_1,c_2)\mm (Y_{1,t_1},Y_{2,t_2})^2  V[Y_{1,t_1},Y_{2,t_2}]$.  Note that $c_1$ is the sum of the  coefficients of some  $Y_{1;e'_1,e'_2,t_1}$, which have  different values.
\hfill\ \end{proof}

\begin{Remark}\label{r'} Suppose that  in the above proposition  char\ $K=p>0$ and in $g$ there exists $a_{e_1,e_2}\not =0$ for some $e_1$, or $e_2$ multiple of $p$. The same proof could give $t_1,t_2$ and  a polynomial $g_1$ such that $g(y_1,y_2)=g_1(y_{1,t_1},y_{2,t_2})$  and $g_1=g_1(0)+c_1  Y_{1,t_1}^{p^{s_1}}+c_2 Y_{2,t_2}^{p^{s_2}}+g_2,$
for some   $c_1,c_2\in V, $ at least one of them nonzero, $s_1,s_2\in {\bf N}$ and $g_2\in  (c_1,c_2)\mm  V[Y_{1,t_1},Y_{2,t_2}]$. However, in $g_2$ could enter also some monomials $ Y_{1,t_1}^{e_1} Y_{2,t_2}^{e_2}$ with $e_1,e_2$ not necessarily multiples of $p$ even $s_1,s_2>1$. We will not use this remark.
\end{Remark}

\begin{Corollary} \label{co'}
Let $V \subset V'$ be an  immediate extension of valuation rings,  $K, K'$ the fraction fields of $V, V'$, $\mm$ the maximal ideal of $V$ and  $y_1,y_2\in K'$ two  elements  which are not in  $K$. 
Assume that char\ $K=p>0$ and $y_i$, $i=1,2$ are   pseudo limits of two pseudo convergent sequences $v_i=(v_{i,j})_{1\leq j<\lambda_i}$, $i=1,2$ over $V$, which have no pseudo limits in $K$.
 Set $y_{i,j}=(y_i-v_{i,j})/(v_{i,j+1}-v_{i,j})$.
Then for every nonzero polynomial  $f\in V[Y_1,Y_2]$  and every two ordinals $\nu_i<\lambda_i$ 
there exist some  $1\leq \nu_i<t_i<\lambda_i$, $i=1,2$ and a polynomial $g\in V[Y_{1,t_1}, Y_{2,t_2}]$
 such that $g=g(0)+c_1  Y_{1,t_1}+c_2  Y_{2,t_2}+g',$ 
for some $c_1,c_2\in V$, at least one of them nonzero,   $g'\in  (c_1,c_2)\mm (Y_{1,t_1},Y_{2,t_2})^2   V[Y_{1,t_1}, Y_{2,t_2}]$ 
and $g(y_{1,t_1},y_{2,t_2})$ is one of the following elements
$f(y_1,y_2)$, \ $y_1f(y_1,y_2)$, \ $y_2f(y_1,y_2)$, \ $y_1y_2f(y_1,y_2)$.
\end{Corollary}

\begin{proof} Fix $\nu_i$, $f=\sum_{e_1,e_2=0}^m a_{e_1,e_2} Y_1^{e_1} Y_2^{e_2} \in V[Y_1,Y_2]$ and let $f=a_{0,0}+f_{00}+f_{10}+f_{01}+f_{11}$ for
 $f_{00}=\sum_{e_1,e_2>0,\ e_1,e_2\ \mbox{not \ multiples\ of\ p}}^m a_{e_1,e_2} Y_1^{e_1} Y_2^{e_2}$, \\
  $f_{10}=\sum_{e_1=0,\ e_2>0,\ e_1\ \mbox{is\ a\  multiple\ of\ p},\ e_2 \mbox{ not\ a \  multiple\ of\ p}}^m a_{e_1,e_2} Y_1^{e_1} Y_2^{e_2}$, \\
 $f_{01}=\sum_{e_1>0,\ e_2=0, \ e_1\ \mbox{not\ a \ multiple\ of \ p},\ e_2\ \mbox{is \ a \ multiple \ of\ p}}^m a_{e_1,e_2} Y_1^{e_1} Y_2^{e_2}$\\
and  $f_{11}=\sum_{e_1,e_2>0,\ e_1,e_2 \ \mbox{multiples\ of\ p}}^m a_{e_1,e_2} Y_1^{e_1} Y_2^{e_2}$.

	Assume that $f_{00}\not =0$.
By Proposition \ref{p'} we see that changing from $Y_1,Y_2$ to some $Y_{1,j_1},Y_{2,j_2} $, $j_1>\nu_1$, $ j_2>\nu_2$ we get $h\in V[Y_{1;j_1},Y_{2;j_2}]$ such that $h(y_{1,j_1},y_{2,j_2})=f(y_1,y_2)$ and $h=h_{00}+{\bar h}$ with ${\bar h}\in V[Y_{1,j_1}^p, Y_{2,j_2}]$, or  ${\bar h}\in V[Y_{1,j_1}, Y_{2,j_2}^p]$  and \\ 
$h_{00}-h_{00}(0)=c'_1Y_{1,j_1}+c'_2Y_{2,j_2}+h'_{00}$ for some $c'_1,c'_2\in V$, at least one of them nonzero, and $h'_{00}\in (c_1',c_2')\mm (Y_{1,j_1},Y_{2,j_2})^2   V[Y_{1,j_1},Y_{2,j_2}]$.

	 If $\min\{\val(c'_1),\val(c'_2)\}<\val({\bar h}(y_{1,j_1},y_{2,j_2})-{\bar h}(0))$ then the conclusion works  for some $c_1,c_2$ with $\val(c_1)=\val(c'_1)$,  $\val(c_2)=\val(c'_2)$. If  $\min\{\val(c'_1),\val(c'_2)\}=\val({\bar h}(y_{1,j_1},y_{2,j_2})-{\bar h}(0))$ then we reduce to the above inequality because we may change from $Y_{1,j_1},Y_{2,j_2} $ to $Y_{1,j'_1},Y_{2,j'_2} $ 
 for some $j'_1>j_1$,  $j'_2>j_2$. Thus the conclusion works again for some $c_1,c_2$.	
Similarly,  if   $\min\{\val(c'_1),\val(c'_2)\}<\val({\bar h}(y_{1,j'_1},y_{2,j'_2})-{\bar h}(0))$
for some $j'_1>j_1$,  $j'_2>j_2$ the conclusion works.

Now assume that   $\min\{\val(c'_1),\val(c'_2)\}>\val({\bar h}(y_{1,j'_1},y_{2,j'_2})-{\bar h}(0))$
for all $j_1<j_1'<\lambda_1$,  $j_2<j_2'<\lambda_2$.  Then we can omit practically $f_{00}$ (this is also the case when  $f_{00}=0$). We may assume that  ${\bar h}\in V[Y_{1,j_1}^p, Y_{2,j_2}]$, then 
 $Y_{1,j_1}{\bar h}$ has the form $\sum_{e_1>0,\ e_2=0,\ e_1 \ \mbox{not \ a\ power\ of\ p}}^{m+1} b_{e_1,e_2} Y_{1,j_1}^{e_1} Y_{2,j_2}^{e_2}$, for some $ b_{e_1,e_2}\in V$, that is of the form $H_1+H_2$ for 
$H_1=\sum_{e_1,e_2>0,\ e_1,e_2\ \mbox{not\ powers \ of\ p}}^{m+1} b_{e_1,e_2} Y_{1,j_1}^{e_1} Y_{2,j_2}^{e_2}$ and
$H_2=\sum_{e_1>0,\ e_2=0,\ e_2\ \mbox{power\ of\ p}}^{m+1} b_{e_1,e_2} Y_{1,j_1}^{e_1} Y_{2,j_2}^{e_2}$.  If $H_1\not= 0$ we may apply Proposition \ref{p'} to $H_1$ and we get a polynomial $H_1'\in V[Y_{1,t_1},Y_{2,t_2}]$ with
$H_1( Y_{1,j_1}, Y_{2,j_2})=H_1'( Y_{1,t_1}, Y_{2,t_2})$ and 
 $H_1'-H_1'(0)=c_1''Y_{1,t_1}+c''_2Y_{2,t_2} +H''_1$ for some $t_i>j_i$, $i=1,2$, $c_1'',c''_2$ in $V$ and $H_1''\in (c''_1,c''_2)\mm V [ Y_{1,t_1}, Y_{2,t_2}]$. By this transformation $H_2$ goes in some $H'_2$.
If $\min\{\val(c''_1),\val(c''_2)\}<\val(H_2'( y_{1,t_1}, y_{2,t_2})-H_2'(0))$ then the conclusion works for some $c_1,c_2$  if we consider $Y_1f$ instead $f$. The same thing is true if the above inequality holds for some $t_1'>t_1$, $t_2'>t_2$. Otherwise, as before we may consider $Y_1 Y_2f$ instead $f$.
\hfill\ \end{proof}

\begin{Lemma} \label{k}
Let $V \subset V'$ be an  immediate extension of valuation rings,  $K, K'$ the fraction fields of $V, V'$, $\mm, \mm'$ the maximal ideals of $V,V'$, 
$(y_e)_{1\leq e\leq  m}$ some elements of $V'$ and $g_e\in V[Y_1,\ldots,Y_m]$, $1\leq e<m$ some polynomials such that the determinant of $((\partial g_e/\partial Y_i)((y_{e'}))_{1\leq e'\leq m, 0\leq i<m}$ is not in $\mm'$.
 Assume that $y_m$ is transcendental over $K$ and $g_e((y_{e'})_{e'})=0$ for all $1\leq e<m$.
 Then $V[y_1,\ldots,y_m]_{\mm'\cap V[y_1,\ldots,y_m]}$ is  a smooth $V$-subalgebra  of $V'$.
\end{Lemma}
\begin{proof} Let $\rho:C=V[Y_1,\ldots,Y_m]/((g_e))\to A=V[y_1,\ldots. y_m]$ be the map given by $Y_e\to y_e$. Then $\Ker \rho\cap V[Y_m]=0$ because $y_m$ is transcendental over $K$. Note that $C_{\rho^{-1}(\mm')}$ is etale over $V[Y_m]$ and so essentially finite (see \cite[Theorem 2.5]{S}). Thus the kernel of the map   
$\rho':C_{\rho^{-1}(\mm')}\to A_{\mm'\cap A}$ induced by $\rho$ must be zero because its intersection with $V[Y_m]$ is zero. So $\rho'$ defines an isomorphism of smooth $V$-algebras. 
\hfill\ \end{proof}

\begin{Lemma} \label{1}
 Let  $ V'$ be an  immediate extension  of a valuation ring $V$ containing a field, $\mm, \mm'$ the maximal ideals of $V,V'$, $K\subset K'$ their fraction field extension and $y\in V'$ an unit, transcendental element over $K$. 
Let  $f\in V[Y]$ be a nonzero polynomial and some $d\in V\setminus \{0\}$ and an unit $z\in V'$ such that $f(y)=d z$. Then there exists 
 a smooth $V$-subalgebra  of $V'$ containing $y,z$.  
\end{Lemma}
\begin{proof} 
By \cite[Theorem 1]{Kap}  $y$ is a pseudo limit of a pseudo convergent sequence $v=(v_j)_{j<\lambda}$, which 
 has no pseudo limits in $K$.  Set $y_j=(y-v_j)/(v_{j+1}-v_j)$. We can assume that $v$ is algebraic, otherwise the result follows for example using \cite[Lemma 15]{P} (this is not important for our proof).

 By Corollary \ref{co} there exist some  $j<\lambda$ and a polynomial $h\in V[Y_j]$ such that either 
$f(y)=h(y_j)$, or $y f(y)=h(y_j)$
 and  $h=h(0)+c  Y_j+h',$
for some   $c\in V\setminus \{0\}$ and $h'\in  c\mm Y_j^2 V[Y_j]$.   If $d|c$ then $d|h(0)$. If $f(y)=h(y_j)$ then  $z=h(y_j)/d$ is contained in the smooth $V$-subalgebra $V[y_j]$ and $y$ is too.
 If $yf(y)=h(y_j)$  then $yz=h(y_j)/d$ and again  $  yz$ is contained in  in the smooth $V$-subalgebra $V[y_j]$ and so $y$ is too. Thus $z\in V[y_j]_{\mm'\cap V[Y_j]}$.

Now assume that $d$ does not divide $c$. Then $c|d$ and $c|h(0)$. If $f(y)=h(y_j)$ then apply Lemma \ref{k} to $g=(h-dZ)/c$ and we see that $y,z$ are contained in the smooth $V$-subalgebra $V[y_j,z]_{\mm'\cap V[y_j,z]}$ of $V'$.  If $yf(y)=h(y_j)$ then apply Lemma \ref{k} to $g=(h-d YZ)/c$ and we see that $y_j$ and  $yz$ are contained in the smooth $V$-subalgebra $V[y_j, yz]_{\mm'\cap V[y_j,yz]}$ of $V'$, which contains also $y$ and so $z$.
 \hfill\ \end{proof}

\begin{Theorem} \label{key}
 Let  $ V'$ be an  immediate extension  of a valuation ring $V$ containing a field, $\mm, \mm'$ the maximal ideals of $V,V'$, $K\subset K'$ their fraction field extension and $y_0\in V'$ an unit, transcendental element over $K$ with $K'=K(y_0)$. 
  Then $V'$ is  a filtered  union of its  smooth  $V$-subalgebras.
\end{Theorem}

\begin{proof} Let $\rho:V[Y_0]\to V'$ be the injective map given by $Y_0\to y_0$ and $f_1,\ldots, f_n\in V[Y_0]$ nonzero polynomials. Then there exist some $d_e\in V\setminus \{0\}$ and some units $y_e\in V'$ such that $f_e(y_0)=d_ey_e$, $1\leq e\leq n$. We claim that there exists 
 a smooth $V$-subalgebra $A_{f_1,\ldots,f_n}$ of $V'$ containing $(y_e)$.  
  By \cite[Theorem 1]{Kap}  $y_e$ is a pseudo limit of a pseudo convergent sequence $v_e=(v_{e,j_e})_{j_e<\lambda_e}$, which has no pseudo limits in $K$,  $0\leq e\leq n$. 
   Set $y_{e,j_e}=(y_e-v_{e,j_e})/(v_{e,j_e+1}-v_{e,j_e})$. 
	
	Apply induction on $n$. If $n=1$ we apply Lemma \ref{1}. 
Assume that $n\geq 2$. Using induction hypothesis on $n$ we may suppose after a change of the numbering of $(y_e)$
that there exist $j_e<\lambda_e$, $0\leq e< n$ and  $h_1,\ldots, h_{n-1}\in V[Y_{0,j_0},\ldots,Y_{n-1,j_{n-1}}]$
such that $h_i((y_{e',j_{e'}}))=0$, $1\leq i< n$ and the determinant of the matrix 
$$((\partial h_i/\partial Y_{e,j_{e}}) ((y_{e',j_{e'}})))_{1\leq i \leq n, 0\leq e\leq n-2}$$ is not in $\mm'.$

As in the proof of Lemma \ref{1} we have to consider two cases. First assume that applying Lemma \ref{1} we arrive in the second case, namely that for some $j_0 <\lambda_0$ there exists $h_1\in V[Y_{0,j_0},Y_1]$ such that
$h_1(y_{0,j_0},y_1)=0$ and $(\partial h_1/\partial Y_{0,j_0})(y_{0,j_0},y_1)\not \in \mm'$.

Note that $y_n$ is transcendental  over $K$ because $y_0$ is so. Since $y_0$ is algebraic over $K(y_n)$  and $y_{n-1}\in K(y_0)$ we see that $y_{n-1}$ is algebraic over $K(y_n)$. Thus there exists a  nonzero polynomial $g_{n-1,n}\in K[Y_{n-1},Y_n]$ with $g_{n-1,n}(y_{n-1},y_n)=0$. Clearly, we may choose $g_{n-1,n}\in V[Y_{n-1}, Y_n]$.

By Corollary \ref{co'}
there exist some  $j_{n-1}<t_{n-1}<\lambda_{n-1}$, $t_n<\lambda_n$ and a polynomial $h_{n-1,n}\in V[Y_{n-1,t_{n-1}}, Y_{n,t_n}]$
 such that  
$$h_{n-1,n}- h_{n-1,n}(0)=c_{n-1} Y_{n-1,t_{n-1}}+c_n  Y_{n,t_n}+h_{n-1,n}'$$
for some $c_{n-1},c_n\in V$, at least one of them nonzero, and \\
 $h_{n-1,n}' \in  (c_{n-1},c_n)\mm V[Y_{n-1,t_{n-1}}, Y_{n,t_n}]$ 
and 
$h_{n-1,n}(Y_{n-1,t_{n-1}}, Y_{n,t_n})$ corresponds to one of the following polynomials $g_{n-1,n}(Y_{n-1},Y_n),$\  \ $Y_{n-1}g_{n-1,n}(Y_{n-1},Y_n)$,\\
 $Y_n g_{n-1,n}(Y_{n-1},Y_n)$,
$Y_{n-1}Y_n g_{n-1,n}(Y_{n-1},Y_n)$.
Set $t_q=j_q$ for $0\leq q\leq n-2$.
The change from $j_e$  to $t_e$ will modify a little the coefficients of $h_i$, $1\leq i<n-1$ but their forms remain.  

Assume that  $c_{n-1}\not =0$  and $\val(c_{n-1})\leq\val(c_n)$ 
then $c_{n-1}|h_{n-1,n}(0)$ because\\
 $h_{n-1,n}(y_{n-1,t_{n-1}},y_{n,t_n})=0$ and we may change 
$h_{n-1,n}$ by $h_{n-1}=h_{n-1,n}/c_{n-1}$ independently if it was obtained from $g_{n-1,n}(Y_{n-1},Y_n)$, or $Y_{n-1}g_{n-1,n}(Y_{n-1},Y_n)$, or $Y_n g_{n-1,n}(Y_{n-1},Y_n)$, or $Y_{n-1} Y_n g_{n-1,n}(Y_{n-1},Y_n)$.  Thus a minor of maximal rank of the matrix 
$((\partial h_i/\partial Y_{e,t_e}) ((y_{e',t_{e'}})))_{1\leq i \leq n, 0\leq e< n}$ is  not contained in $\mm'$.
  By   Lemma \ref{k} we get a smooth $V$-subalgebra
 $A_{f_1,\ldots,f_n}= V[y_{0,t_0},\ldots, y_{n,t_n}]_{\mm'\cap V[y_{0,t_0},\ldots, y_{n,t_n}}$  of $V'$ containing $(y_{e,t_e})$ and so all $y_e$.
Certainly, we could have above also $c_n\not =0$ and $\val(c_{n-1})>\val(c_n)$
in which case change  $f_n$ with $f_{n-1}$. Hence our claim is proved.

Now assume that applying Lemma \ref{1} we arrive in the first case of the proof, namely $y_1\in V[y_{0,j_0}]_{\mm'\cap V[y_{0,j_0}]}$ for some $j_0<\lambda_0$. Then we may replace $y_0$ by
$$(v_{0,j_0+1}-v_{0,j_0})y_0+v_{0,j_0}$$
and we may omit $y_1$, that is we denote $y_e$ by $y_{e-1}$ for $1<e\leq n$. So we arrive in the case $n-1$ when we apply the induction hypothesis to show our claim.

The family  $(A_{f_1,\ldots, f_n})$ given by all finite subsets $\{f_1,\ldots,f_n\}$ of nonzero polynomials of $V[Y_0]$ is filtered and consider 
 their union $\mathcal A$. We claim that $V'={\mathcal A}$. Indeed, 
 let $f_1(y_0)/f_2(y_0)$ be a rational function from $V'$ (a general element of $V'$ has the form $f_1(y_0)/f_2(y_0)$ with $f_i\in V[Y_0]$ and 
$\val(f_1(y_0))\geq\val(f_2(y_0))$).  We can assume  $f_i(y_0)=d_iy_i$, $i=1,2$, for some $d_i \in V\setminus \{0\}$ and some units $y_i\in V'$, $i=1,2$. For some $t_e<\lambda_e$, $0\leq e\leq 2$ we find as above a smooth $V$-subalgebra $A_{f_1,f_2}$ of $V'$ containing   $y_{0,t_0}$, $y_{1,t_1}$ , $y_{2,t_2}$ and so containing
 $(y_e)$,  $0\leq e\leq 2$.
 As  
$\val(d_1)\geq \val(d_2)$ because $f_1(y_0)/f_2(y_0)\in V'$ we get  
$$f_1(y_0)/f_2(y_0)=(d_1/d_2)y_1y_2^{-1}\in   A_{f_1,f_2}\subset {\mathcal A}.$$
\hfill\ \end{proof}

{\bf Proof of Theorem \ref{T3}}

 Firstly assume  $K'/K$ is a field extension  of finite type, let us say $K'=K(x_1,\ldots,x_n)$ for some unit elements $x=(x_1,\ldots,x_n)$ of $V'$ algebraically independent over $K$. Set $V_i=V'\cap K(x_1,\ldots,x_i)$, $1\leq i< n$ and $V_0=V$, $V_n=V'$. 
Then $V_{i+1}$ is a   filtered  union of its smooth  $V_i$-algebras  for all $0\leq i<n$ by Theorem \ref{key}  and we are done. In general, express $K'$ as a filtered  union of some subfields $(K_i)_i$ of $K'$ which are pure transcendental, finite type field extensions of $K$. Then $V'$ is a filtered increasing union of $V'\cap K_i$.

\begin{Corollary} \label{C2}
Let $V\subset V'$ be an  immediate extension of valuation rings  and $K\subset K'$ its fraction field extension.  Assume that $V$ contains a field,  $V$ is Henselian and $K'=K(x)$ for some algebraically independent elements $x$ over $K$. 
 Then any finite
system of polynomials over $V$, which has a solution in $V'$, has also one in $V$.
\end{Corollary}
 For the proof apply \cite[Proposition 18]{P1} after using Theorem \ref{T3}.

{\bf Proof of Corollary \ref{C}}

 By \cite[Lemma 27 (1)]{P} we see that $ V$ is an immediate extension of a valuation ring $ V_0$ which is a filtered  union of localizations of its polynomial $ k$-subalgebras (see \cite[Theorem 1, in VI (10.3)]{Bou}, or \cite[Lemma 26 (1)]{P}).  In fact, important here is that there exists a cross-section $s:\Gamma\to K^*$, since $\Gamma$ is free and  we take $V_0=V\cap k(s(\Gamma))$ (a {\em cross-section}
of $V$ is a section
in the category of abelian groups of the valuation map $\val : K^* \to \Gamma$). In our case, $V_0=V\cap k(y)$ and $s$ is defined by $\val(y)\to y$. 
Now it is enough to apply Theorem \ref{T3}
 for $V_0\subset V$.

\end{document}